\documentclass{amsart}
\usepackage{amsthm,amssymb,mathrsfs,calc,graphicx}

\def\bC{\mathbf{C}}
\def\bbbp{\mathbb{P}}

\def\bbbl{\mathbb{L}}
\def\bbbm{\mathbb{M}}

\def\YY{\mathbb{Y}}
\def\Esp{\mathbb{E}}
\def\Var{\mathbb{V}}

\def\Bcal{\mathcal{B}}
\def\Ccal{\mathcal{C}}
\def\Dcal{\mathcal{D}}

\def\Lcal{\mathcal{L}}

\def\Ncal{\mathcal{N}}

\def\Tcal{\mathcal{T}}

\def\Xcal{\mathcal{X}}

\def\bbbr{\mathbb{R}}
\def\bbbc{\mathbb{C}}
\def\bbbn{\mathbb{N}}
\def\N{{N}}
\def\P{\mathbb{P}}
\def\d{\mathrm{d}}
\def\Lip{\text{Lip}}
\def\XX{\mathbb{X}}
\def\bC{\mathbf{C}}
\def\mmmd{{\Delta}}

\newtheorem{theo}{Theorem}[section]

\newtheorem{lem}[theo]{Lemma}
\newtheorem{rem}[theo]{Remark}
\newtheorem{cor}[theo]{Corollary}

\begin{document}

\title[Stein's method with Papangelou intensity]{Stein's method and Papangelou intensity  for
    Poisson or Cox process approximation}

  \author{Laurent Decreusefond}
  \author{Aur\'elien Vasseur}
\address{LTCI, Telecom ParisTech, Universit\'e Paris-Saclay, 75013, Paris,
  France}
\email{laurent.decreusefond@telecom-paristech.fr, aurelien.vasseur@telecom-paristech.org }

\begin{abstract}
  In this paper, we apply the Stein's method in the context of point processes,
  namely when the target measure is the distribution of a finite Poisson point
  process. We show that the so-called Kantorovich-Rubinstein distance between
  such a measure and another finite point process is bounded by the
  $L^1$-distance between their respective Papangelou intensities. Then, we
  deduce some convergence rates for sequences of point processes approaching a
  Poisson or a Cox point process.
\end{abstract}


\maketitle{}

\section{Introduction} \label{sec_introduction}

A fruitful way to get some approximations in probability theory is drawn from
the Stein's method, introduced in 1972 by Stein \cite{stein_bound_1972} in order
to give a convergence speed for the Central Limit Theorem. The approach was soon
extended to the Poisson distribution by Chen
\cite{chen_poisson_1975}. There are too many articles about the development and
applications of this method for them to be cited. If we restrict our attention
to point processes, we must cite  \cite{barbour_steins_1988, barbour_steins_1992,
  barbour_steins_2006} about Poisson point process approximation. In this series
of articles, the distance between a generic point process and a Poisson point process is
stated in terms of the Palm measure of the tested process or in terms of its
compensator when processes on the half line are considered. After the pioneering
work \cite{nourdin_normal_2012}, it became evident that the Malliavin calculus
is one tool of choice to systematize the Stein's approach. In this respect,
\cite{decreusefond_upper_2010} explores the link between the different notion of
Malliavin gradients on the Poisson space and the different notions of distance
between distribution of point processes (see also
\cite{schuhmacher_new_2008}). In \cite{decreusefond_functional_2016}, it is
shown that some transformations of a Poisson point process converge to a Poisson
process on the real line using tools of the Poisson-Malliavin calculus. Some other
applications of this framework are given in
\cite{last_normal_2016,schulte_distances_2014}. As we will see below, the
Stein's method is based on a representation of the target measure as the
invariant and stationary measure of a configuration-valued Markov process. The existence of such a
process is absolutely not granted so there is a very few results
about convergence to processes which are not Poissonian. The paper
\cite{schuhmacher_gibbs_2014}  by Schumacher
and Stucki, which addresses the convergence towards a Gibbs point process
then  appears as an exception.

In this article, we develop a variant of the Stein's method in order to evaluate
the distance between the distribution of some point processes and that of a
Poisson point process. In what follows, we abuse slightly the vocabulary by
identifying a point process as a random variable and its distribution. Let
$\zeta_M$ denote a Poisson process with finite intensity measure $M$ on a space
$\XX$. The first step of the Stein's method consists in characterizing the target
object, here the distribution of  $\zeta_M$. We must first construct a
functional operator $L$ which satisfies 
\begin{equation*}
  \Bigl( \Esp[LF(\Phi)]=0,\ \forall F\in \mathcal F \Bigr) \Longleftrightarrow \Phi\stackrel{\text{dist}}{=}\zeta_M,
\end{equation*}
where $\mathcal F$ is a rich enough class of test functions. For $\zeta_M$, it is
known that we can take
\begin{equation}\label{eq_Article_Stein:1}
  LF(\phi)=\int_{\XX} (F(\phi+x)-F(\phi)) M(\d x)+\sum_{y\in\phi} (F(\phi\setminus y)-F(\phi)).
\end{equation}
The second step is to solve the so-called Stein's equation, that is to find, for
any test function $F:\widehat{N}_\XX\to\bbbr$, a function
$H_F:\widehat{N}_\XX\to\bbbr$ such that, for any $\phi\in\widehat{N}_\XX$,
\begin{equation*}
  LH_F(\phi)=\Esp[F(\zeta_M)]-F(\phi).
\end{equation*}
We here use the so-called generator approach (see \cite{reinert_three_2005} for a
survey and also~\cite{coutin_steins_2013,shih_steins_2011}). In the
current situation, the Markov process with values in the space of configurations
$\widehat{N}_\XX$, 
which has invariant and stationary measure $\zeta_M$ is called the Glauber
process. If $(P_t)_{t\geq0}$ is the associated semi-group, one can show that, for any $\phi\in \widehat{N}_\XX$,
\begin{equation*}
  LH_F(\phi)=\int_0^{+\infty}LP_sF(\phi)\d s,
\end{equation*}
which leads to the so-called Stein-Dirichlet representation formula:
\begin{equation*}
  \Esp[F(\zeta_M)]-F(\phi)=\int_0^{+\infty}LP_sF(\phi)\d s.
\end{equation*}

The Kantorovich-Rubinstein (or Wasserstein-1) distance between a point
process $\Phi$ and $\zeta_M$ is defined as
\begin{equation*}
  \sup_{F\in \Lip_{1}(\widehat{N}_{\XX}) }\Esp[F(\zeta_M)]-\Esp[F(\zeta_M)].
\end{equation*}
Using \eqref{eq_Article_Stein:1}, the next step is then to transform the
rightmost term
\begin{equation*}
  \Esp\left[   \sum_{y\in \Phi} \left( P_{s}F(\Phi\setminus y)-P_{s}F(\Phi) \right) \right]
\end{equation*}
into the expectation of an integral with respect to the measure $M$ plus a
remainder term. To do so in  the present article, our
strategy is based on the use of the Papangelou intensity instead of the Palm
measure as in \cite{barbour_steins_1992}. Although the link
between Papangelou intensity and Palm theory is strong, our choice offers the
advantage of dealing with  a function rather than  a probability measure. The
calculations are then easier. Intuitively, if $c$ denotes the
Papangelou intensity of a given point process $\Phi$ (with respect to a given
Radon measure $\ell$), the probability of finding a particle in the location $x$
given that there is a particle located at each point of the configuration (or
locally finite subset) $\phi$ is represented by the quantity $c(x,\phi)$. For a
Poisson point process $\zeta_{M}$ with intensity $M(\d x)=m(x)\ell(\d x)$, the
Papangelou intensity is simply given by
\begin{equation*}
  c(x,\phi)=m(x).
\end{equation*}
This corroborates  the idea that a given particle of a Poisson point
process does not depend on the other particles of the configuration.

Among the asymptotic behaviors we consider in this work, several of them 
concern transformations of  point processes. More specifically,
a way to transform a point process into a new point process with less
interactions between its particles (in other words to reduce its level of
repulsiveness or attractiveness) is to use operations which insert some
independence. We will in particular focus on two transformations: independent
superposition and independent thinning. 

One important part of our contributions consists in new
results about the Papangelou intensities, stating how it is transformed by reduction to a
compact,  independent superposition or independent thinning and 
rescaling.

The paper is organized as follows. In Section \ref{sec_preliminaries}, the main
elements from the point process theory are presented more formally. In Section
\ref{sec_steins_method}, we present in details the development of the Stein's
method when the target measure is a  finite Poisson point process, state the so-called Stein-Dirichlet
representation formula and obtain an upper bound for the Kantorovich-Rubinstein
distance associated to the total variation distance between a finite Poisson
point process and another finite point process. We give in Section
\ref{sec_papangelou} the elements concerning Papangelou intensities which will
be necessary to state some convergence results in the next section, in
particular the definition of weak repulsiveness. From all these preliminary
results, we deduce in Section \ref{sec_applications} some convergence rates when
considering Kantorovich-Rubinstein distance between Poisson or Cox point
processes and other point processes, which are Poisson-like point processes in
Subsection \ref{subsec_applicationpoissonlike}, repulsive point processes in
Subsection \ref{subsec_applicationrepulsive} and thinned point processes in
Subsection \ref{subsec_kallenberg}, where a convergence rate is provided for the
aforementioned theorem from Kallenberg \cite{kallenberg_random_1983}. Proofs are
given in Section \ref{sec_proofs}.

\section{Preliminaries} \label{sec_preliminaries}

\subsection{Notations} \label{subsec_notations}

In this work, we use classical mathematical notations. In particular, $\bbbn$
denotes the space of positive integers, $\bbbn_0$ the space of non-negative
integers, $\bbbr$ the space of real numbers and $\bbbc$ the space of complex
numbers.

We consider a locally compact metric space $(\XX,\mmmd_\XX)$ endowed with its
Borel $\sigma$-algebra $\Xcal$ and a (not necessarily diffuse) Radon measure
$\ell$ on $\XX$. The family of relatively compact Borel sets is denoted by
$\Xcal_0$. A distance $\mmmd$ on a space $\YY$ will be denoted $\mmmd_{|\YY}$ if
necessary. 

The set of bounded measurable functions from $\XX$ to $\bbbr_+$ with compact
support is denoted $\Bcal_+(\XX)$. If $f$ is a function from $\XX$ to $\bbbc$,
then $\|f\|_\infty$ denotes the supremum of the set $\{|f(x)|:~x\in\XX\}$. For
$p\in[1,+\infty)$, $L^p(\XX,\ell)$ denotes the space of functions
$f:\XX\to\bbbc$ such that $|f|^p$ is integrable with respect to $\ell$. The
space of continuous functions from $\XX$ to $\bbbr$ (respectively $\bbbc$) with
compact support is denoted $\Ccal_K(\XX,\bbbr)$ (respectively
$\Ccal_K(\XX,\bbbc)$). The integral of an integrable function $f$ with respect
to $\ell$ will be more simply written $\int_\XX f(x)\d x$ when there is no
ambiguity.

The space of measures on $(\XX,\Xcal)$ will be denoted $\bbbm$, $\bbbm_R$ is the
space of Radon measures on $\XX$ and $\bbbm_1$ the family of all probability
measures on $\XX$. The space of measures on $\bbbm$ and the space of probability
measures on $\bbbm$ will be respectively denoted by $\bbbm'$ and $\bbbm_1'$. For
any $x\in\XX$, $\delta_x$ denotes the Dirac measure at $x$. For any
$A\in\Xcal$, $\ell(A)$ may also be denoted $|A|$. For any $A\subset\XX$,
$\textbf{1}_A$ denotes the indicator function of the subset $A$ of $\XX$. If
$\varphi\in\bbbm$ and $m$ is a function on $\XX$ integrable with respect to
$\varphi$, then $\langle m,\varphi \rangle$ denotes the integral of $m$ with
respect to $\varphi$, the measure $\nu$ with density $m$ with respect to
$\varphi$ is denoted $m\varphi$. In this case, $m$ is denoted
$\frac{\d\nu}{\d\varphi}$ and the fact that $\nu$ is absolutely continuous with
respect to $\varphi$ is denoted $\nu\ll\varphi$.

For any random element $X$ of $\XX$, $\P_X$ denotes the probability distribution
of $X$. If $\mmmd$ is a distance on $\bbbm_1$ and $X_1,X_2$ two random elements
of $\XX$ with respective distributions $\P_1,\P_2$, then we will also write
$\mmmd(X_1,X_2)$ instead of $\mmmd(\P_1,\P_2)$.


A {counting measure} $\xi$ on $\XX$ is a measure on $\XX$ such that, for any
$A\in\Xcal_0$,

\begin{equation*}
  \xi(A)\in\bbbn_0.
\end{equation*}

A {configuration} (respectively {finite configuration}) on $\XX$ is a locally
finite (respectively finite) counting measure on $\XX$.

The space of configurations on $\XX$ is denoted $\N_\XX$ and $\widehat{\N}_\XX$
denotes the space of finite configurations on $\XX$. We endow $\N_\XX$ with
$\Ncal_\XX$ defined as the smallest $\sigma$-algebra on $\N_\XX$ such that
$\phi\in\N_\XX\mapsto \phi(A)$ is measurable for any $A\in\Xcal_0$. The
restriction of $\Ncal_\XX$ to $\widehat{\N}_\XX$ is denoted
$\widehat{\Ncal}_\XX$.

Note that, although a configuration is defined as a measure, it also may be
considered as a locally finite subset of points, and will then often be
introduced with set-theoretic notations.

For a function $F:N_\XX\to\bbbr$ which is integrable with respect to $\omega$,
the integral $\int_\XX F(x)\omega(\d x)$ will be often denoted
$\sum_{x\in\omega}F(x)$. For any $k\in\bbbn$ and $x_1,\dots,x_k\in\XX$,
$\{x_1,\dots,x_k\}$ also denotes the configuration $\omega$ defined, for any
$A\in\Xcal$, by
\begin{equation*}
  \omega(A)=\#\big\{ i\in\{1,\dots,k\},x_i\in A\big\}
\end{equation*}
where, for any finite set $B$, $\#B$ denotes the number of elements in $B$. By a
slight abuse of notation, we will also denote $F(x_1,\dots,x_k)$ instead of
$F(\{x_1,\dots,x_k\})$.

A function $f:\N_\XX\to\bbbr$ is said to be {increasing} (resp. {decreasing})
if, for any $\phi_1,\phi_2\in\N_\XX$,
\begin{equation*}
  (\phi_1\subset\phi_2)\implies(f(\phi_1)\leq f(\phi_2))\ [\text{resp.}\ (\phi_1\subset\phi_2)\implies(f(\phi_1)\geq f(\phi_2))].
\end{equation*}
A subset $A$ of $\N_\XX$ is said to be {increasing} (resp. {decreasing}) if
$\textbf{1}_A$ is increasing (resp. decreasing), that is, if for any $\phi_1\in
A$ and $\phi_2\in N_\XX$,
\begin{equation*}
  (\phi_1\subset\phi_2)\implies(\phi_2\in A)\ [\text{resp.}\  (\phi_2\subset\phi_1)\implies(\phi_2\in A)].
\end{equation*}
\subsection{Point processes} \label{subsec_point_processes}

Point processes are formally seen as random locally finite subsets of points and
provide a powerful mathematical tool with some applications in many areas, such
as forestry \cite{stoyan_recent_2000}, astronomy \cite{babu_spatial_1996} or
epidemiology \cite{gatrell_spatial_1996}, telecommunications and precisely
wireless networks \cite{baccelli_stochastic_2010}, and more generally in each
field where the spatial  distribution some particles needs to be analyzed in a
mathematical way.

The following notions concern point process theory and come
essentially from \cite{daley_introduction_2003}.

A {point process} $\Phi$ on $\XX$ is a random configuration on $\XX$. Its
{intensity measure} is the measure $M$ on $(\XX,\Xcal)$ defined, for any
$A\in\Xcal$, by
\begin{equation*}
  M(A)=\Esp[\Phi(A)].
\end{equation*}
To describe the distribution of a point process, several characterizations are
available. Among them, the Laplace functional
$\Lcal_\Phi:\Bcal_+(\XX)\to\bbbr_+$ is given for any $f\in \Bcal_+(\XX)$ by
\begin{equation*}
  \Lcal_\Phi(f)=\Esp\Big[\exp\Big(-\int_\XX f(x)\Phi(\d x)\Big)\Big].
\end{equation*}
and offers the advantage to give easily some information on point processes built by
independent superposition or other independent transformations. It is
mathematically speaking the most efficient but is not the most intuitive. It may
be useful to characterize a point process by considering other functionals such
as its Janossy function and correlation function.

If $x_1,\dots,x_n$ are $n$ particles in $\XX$, the Janossy function $j$ is
defined in such a way that $j(x_1,\dots,x_n)$ intuitively represents  the
probability of finding exactly $n$ particles in the vicinity of  $x_1,\dots,x_n$,
while the correlation function $\rho$ is such that $\rho(x_1,\dots,x_n)$
represents the probability of finding at least $n$ particles in the neighborhood
of  $x_1,\dots,x_n$,
with possibly some  other particles at other locations.

Formally, if $\Phi$ is almost surely finite, its {Janossy function} (with
respect to $\ell$) $j:\widehat{\N}_\XX\to\bbbr_+$ is then defined, if it exists,
for any measurable function $u:\widehat{\N}_\XX\to\bbbr_+$ by

\begin{equation*}
  \Esp[u(\Phi)]=\displaystyle\sum_{k=0}^{+\infty} \dfrac{1}{k!}\int_{\XX^k} u(x_1,\dots,x_k)j(x_1,\dots,x_k) \ell(\d x_1) \dots \ell(\d x_k)
\end{equation*}
and its {correlation function} (with respect to $\ell$)
$\rho:\widehat{\N}_\XX\to\bbbr_+$ is defined, if it exists, for any measurable
function $u:\widehat{\N}_\XX\to\bbbr_+$ by

\begin{equation*}
  \Esp\Big[\displaystyle\sum_{\substack{\alpha\in\widehat{\N}_\XX \\ \alpha\subset\Phi}}u(\alpha)\Big]=\displaystyle\sum_{k=0}^{+\infty} \dfrac{1}{k!}\int_{\XX^k} u(x_1,\dots,x_k)\rho(x_1,\dots,x_k) \ell(\d x_1) \dots \ell(\d x_k).
\end{equation*}

Correlation function also provides a way to specify the repulsiveness or
attractiveness of a point process, which will be considered as repulsive
(respectively attractive) as soon as, for any $x,y$,

\begin{equation} \label{repulsivenesscorrelation} \rho(x,y)\leq\rho(x)\rho(y)\
  [\text{resp.}\ \rho(x,y)\geq\rho(x)\rho(y)].
\end{equation}



Another typical functional both provides a way to characterize a point process
and an intuitive interpretation: introduced in 1974 by Papangelou
\cite{papangelou_conditional_1974}, the Papangelou intensity $c$ of a point
process $\Phi$ relies on the so-called {reduced Campbell measure} $C$, defined
on the product space $(\XX\times\N_\XX,\Xcal\otimes\Ncal_\XX)$ for any
$A\in\Xcal\otimes\Ncal_\XX$ by
\begin{equation*}
  C(A)=\Esp\Big[\displaystyle\sum_{x\in\Phi}\textbf{1}_A(x,\Phi\setminus x)\Big].
\end{equation*}

If $C\ll\ell\otimes\P_\Phi$, any Radon-Nikodym density $c$ of $C$ relative to
$\ell\otimes\P_\Phi$ is then called (a version of) the {Papangelou intensity} of
$\Phi$. More explicitly $c$ is a Papangelou intensity of $\Phi$ if, for any
measurable function $u:\XX\times\N_\XX\to\bbbr_+$,
\begin{equation*} \label{GNZ}
  \Esp\Big[\displaystyle\sum_{x\in\Phi}u(x,\Phi\setminus x)\Big]=\int_\XX
  \Esp[c(x,\Phi)u(x,\Phi)]\ell(\d x).
\end{equation*}
Hence $c(x,\phi)$ represents the probability of finding a particle in the
vicinity of $x$ given that there is a particle located at each point of the
configuration $\phi$. In particular, this leads to consider the variations of
this quantity when the configuration $\phi$ increases: if $\omega\subset\phi$
implies that
\begin{equation} \label{repulsivenesspapangelou} c(x,\phi)\leq c(x,\omega)\
  [\text{resp.}\ c(x,\phi)\geq c(x,\omega)]
\end{equation}
then it rather exhibits repulsiveness (resp. attractiveness).

The Papangelou intensity $c$ of a finite point process on $\XX$ may furthermore
be linked respectively to its Janossy function $j$ and its correlation function
$\rho$ in the following ways: for any $x\in\XX$,

\begin{equation} \label{Papangelouprop2} \Esp[c(x,\Phi)]=\rho(x)
\end{equation}
and, if $\{j=0\}$ is an increasing set, then, for any $x\in\XX$ and $\phi\in
N_\XX$,

\begin{equation} \label{Papangelouprop1}
  c(x,\phi)=\dfrac{j(x\phi)}{j(\phi)}\textbf{1}_{\{j(\phi)\neq0\}}.
\end{equation}

\subsection{Poisson-based point processes} \label{subsec_poisson_based}

We recall here the definitions and some basic properties of some Poisson-based
point processes, in particular their Laplace functionals, Janossy functions and
correlation functions.

The most elementary point process is the {binomial point process}, depending on
a fixed parameter $N\in\bbbn_0$ and a probability measure $\mu$ on $\XX$. Such a
point process has exactly $N$ points drawn independently according to $\mu$.



If $M$ is a Radon measure on $\XX$, the {Poisson point process} (PPP) $\Phi$
with intensity measure $M$ is defined as the unique point process on $\XX$ with
intensity measure $M$ such that, for any disjoint relatively compact subsets
$\Lambda_1,\Lambda_2$, the random variables $\Phi(\Lambda_1)$ and
$\Phi(\Lambda_2)$ are independent. A Poisson point process on $\XX$ with finite
intensity measure $M$ may also be defined as a finite point process $\Phi$ on
$\XX$ such that its total number of points $N$ has a Poisson distribution with
parameter $M(\XX)$ and, conditionally to $N$, $\Phi$ is a binomial point process
on $\XX$ with parameter $N$ and supported by $\frac{M(\cdot)}{M(\XX)}$. Its
Laplace functional is given for any $f\in \Bcal_+(\XX)$ by

\begin{equation*}
  \Lcal_\Phi(f)=\exp\Big\{-\int_\XX (1-e^{-f(x)})M(\d x)\Big\}.
\end{equation*}

If $M(\d x)=~m(x)\d x$, its correlation function $\rho$ is given for any
$\phi\in\widehat{\N}_\XX$ by

$$\rho(\phi)=\displaystyle\prod_{x\in\phi}m(x);$$

and, if $M(\XX)<+\infty$, its Janossy function $j$ is given for any
$\phi\in\widehat{\N}_\XX$ by

$$j(\phi)=e^{-M(\XX)}\displaystyle\prod_{x\in\phi}m(x).$$


A Poisson point process $\Phi$ verifies the Mecke formula: for any measurable
function $u:\XX\times N_\XX\to\bbbr_+$,

\begin{equation*}
  \Esp\Big[\displaystyle\sum_{x\in\Phi}u(x,\Phi\setminus x)\Big]=\int_\XX \Esp[u(x,\Phi)]M(\d x).
\end{equation*}

Conversely, if $M$ is locally finite and $\Phi$ such that, for any $x\in\XX$,
$\Phi(\{x\})\leq1$ a.s., then $\Phi$ is a Poisson point process with intensity
measure $M$.


The Poisson point process may be characterized as the only point process with no
interactions between its particles, that is, without repulsiveness or
attractiveness. For this point process, each compact subset has a
Poisson-distributed number of particles and the respective numbers of particles
in two disjoint compact subsets are independent. It verifies the equality in
(\ref{repulsivenesscorrelation}) and (\ref{repulsivenesspapangelou}) and may be
in this sense considered as the "zero" point process between repulsive and
attractive point processes.

Another crucial property of the space of Poisson point processes is its
stability for independent superposition and thinning: in other words, if
$\Phi_1,\dots,\Phi_n$ are $n$ Poisson point processes with respective intensity
measures $M_1,\dots,M_n$ and $\beta_1,\dots,\beta_n\in[0,1]$, then
$\beta_1\circ\Phi_1+\dots+\beta_n\circ\Phi_n$ is a Poisson point process with
intensity measure $\beta_1M_1+\dots+\beta_nM_n$. In particular, a Poisson point
process $\Phi$ verifies the following invariance property: for any $t\in[0,1]$,

\begin{equation*}
  t\circ\Phi^{(1)}+(1-t)\circ\Phi^{(2)}\overset{\Dcal}{=}\Phi,
\end{equation*}

where $\Phi^{(1)}$ and $\Phi^{(2)}$ are independent copies of $\Phi$.



Poisson point processes are also used as a basis for the definition of some
other categories of point processes, that will be called Poisson-like point
processes.



Among them, a {Cox point process} \cite{cox_point_1980} directed by a random
measure $M$ on $\XX$ is a point process $\Phi$ such that, conditionally to $M$,
$\Phi$ is a Poisson point process with intensity measure $M$; it provides a
useful tool to model attractive distribution of particles.



A {purely random point process} (PRPP) \cite{matthes_infinitely_1978} supported
by a probability measure $\mu$ on $\XX$ and a sequence
$(p_n)_{n\in\bbbn_0}\subset\bbbr_+$ such that $\sum_{n=0}^{+\infty}p_n=1$ is a
finite point process $\Phi$ on $\XX$ such that its number of points $N$ in $\XX$
verifies for any $n\in\bbbn_0$
$$\P(N=n)=p_n;$$ 
and, conditionally to $N$, $\Phi$ is a binomial point process on $\XX$ with
parameter $N$ and supported by $\mu$. In particular, a binomial point process is
a purely random point process with a deterministic number of points, and a finite
Poisson point process is a purely random point process whose number of points
has a Poisson distribution.

Let consider now $C\in\Ncal_\XX$, a Poisson point process $\Phi$ with intensity
measure $M$ and a sequence $(\Phi^{(n)})_{n\in\bbbn}$ of independent copies of
$\Phi$. The point process $\Phi_C$ defined as
\begin{equation*}
  \Phi_C:=\Phi^{(n)}\ \text{if}\ \Phi^{(n)}\in C\ \text{and},\ \forall i\in\{1,\dots,n-1\},\ \Phi^{(i)}\notin C
\end{equation*}
is called the {conditional Poisson point process} associated to $\Phi$ with
intensity measure $M$ and condition $C$.

In particular, the conditional Poisson point process $\Phi_R:=\Phi_{C_R}$
associated to a Poisson point process $\Phi$ on the metric space
$(\XX,\mmmd_\XX)$ with finite parameter measure $M$ and condition

\begin{equation*}
  C_R:=\{\phi\in\widehat{N}_\XX:\forall x,y\in\phi,x\neq y\implies \mmmd_\XX(x,y) \geq R\}\ \text{where}\ R>0
\end{equation*}

is a {hardcore (conditional) Poisson point process}
\cite{haenggi_stochastic_2012} with parameter $R$.


The conditional Poisson point process $\Phi_N:=\Phi_{C_N}$ associated to a
Poisson point process $\Phi$ on $\XX$ with finite parameter measure $M$ and
condition
\begin{equation*}
  C_N:=\{\phi\in\widehat{N}_\XX:\phi(\XX)\leq N\}\ \text{where}\ N\in\bbbn_0.
\end{equation*}
is called a {bounded (conditional) Poisson point process} with parameter $N$.




Gibbs point processes are repulsive point processes, especially in the sense
given by (\ref{repulsivenesspapangelou}), and were introduced in the field of
statistical physics \cite{ruelle_statistical_1969}. The repulsiveness of a Gibbs
point process appears naturally in the expression of its total potential energy
$U$, defined as

$$U(x_1,\dots,x_n)=\displaystyle\sum_{r=1}^n\displaystyle\sum_{1\le i_1<\dots<i_r\le n} \Psi_r(x_{i_1},\dots,x_{i_r}),$$

where $\Psi_r:\XX\to\bbbr_+$ is a measurable and symmetric function, called
{$r^{\text{th}}$-order interaction potential}, which quantifies the degree of
repulsion between $r$ given particles.

In a formal way, a point process on $\XX$ is said to be a {Gibbs point process}
with temperature parameter $\theta>0$ and total potential energy $U$ if its
Janossy function $j$ is given for any $\phi\in\widehat{\N}_\XX$ by

$$j(\phi)=C(\theta)e^{-\theta U(\phi)},$$

for some partition function $C(\theta)>0$.

In light of this classification, a Poisson point process $\Phi$ may be seen at
once as a Cox point process directed by a deterministic measure, a conditional
Poisson point process with condition $C={N}_\XX$; if $\Phi$ is finite, as a
purely random point process supported by a Poisson distribution, and as a Gibbs
point process such that its total potential energy $U$ equals to its
$1^{\text{st}}$-order interaction potential $\Psi_1$.

\subsection{$\alpha$-determinantal/permanantal point
  processes} \label{subsec_determinantal}


Another useful model for the distribution of particles with some repulsion
(called fermion particles in the literature) appears with determinantal point
processes, introduced by Macchi in 1975 \cite{macchi_coincidence_1975}, and
whose mathematical structure were studied in details by Soshnikov
\cite{soshnikov_determinantal_2000}, Shirai and Takahashi
\cite{shirai_random_2003}, then Hough et al. \cite{hough_determinantal_2006}. As
mentioned above, the repulsive behavior of the particles from such point
processes may be intuitively interpreted by observing the correlation function
and the Papangelou intensity. The correlation function $\rho$ of a determinantal
point process is defined as

\begin{equation}
  \rho(x_1,\dots,x_n)=\det(K(x_i,x_j))_{1\leq i,j\leq n},
\end{equation}

where $K$ is the kernel of a bounded symmetric and Hilbert-Schmidt integral
operator on $L^2(\XX,\ell)$, also denoted $K$, i.e., for any $x\in\XX$,

\begin{equation*}
  Kf(x)=\int_\XX K(x,y)f(y)\ell(\d y).
\end{equation*}

and we suppose that the map $K$ is an Hilbert-Schmidt operator from
$L^2(\XX,\ell)$ into $L^2(\XX,\ell)$ which satisfies the following conditions:

\begin{itemize}
 
\item The spectrum of $K$ is included in $[0,1)$.

\item The map $K$ is locally trace-class, i.e., for all compact
  $\Lambda\subset\XX$, the restriction $K_\Lambda=P_\Lambda K P_\Lambda$ of $K$
  to $L^2(\Lambda,\ell_{|\Lambda})$ is trace-class.

\end{itemize}

From this definition, as expected, we can show that
(\ref{repulsivenesscorrelation}) holds for determinantal point processes. Moreover, Georgii and Yoo
\cite{georgii_conditional_2005} provided an explicit expression for the
Papangelou intensity $c$ of such a process (see 
Section \ref{sec_papangelou}) and showed in particular that 
(\ref{repulsivenesspapangelou}) also holds.

%
%
%

It was actually shown \cite{shirai_random_2003} that determinantal point
processes (DPP for short) may be included in a wider class of point processes, called
$\alpha$-determinantal and permanantal point processes, where determinant is
replaced by $\alpha$-determinant and where the coefficient $\alpha$ provides an
indication on the repulsive or attractive nature of the point process. More
precisely, its particles exhibit repulsiveness as soon as $\alpha<0$ (in
particular $\alpha=-1$ corresponds to determinantal point process) and
attractiveness when $\alpha>0$. The attractive point processes of this last
category provide a model in statistical physics for the distribution of boson
particles. The case $\alpha=0$ leads to the Poisson point process, which
consolidates the idea expressed previously of being represented as a "zero"
point process. The reader may also consult Decreusefond et al.
\cite{camilier_quasi-invariance_2010,decreusefond_determinantal_2016} for
surveys and Lavancier et al. \cite{lavancier_determinantal_2015} for statistical
inference on determinantal point processes. It nevertheless can be worthwhile to
cite the following theorem which states the orthonormal decomposition of a
sufficiently regular integral operator.
\begin{theo}\label{DPPprop1} 

  Let $K:L^2(\XX,\ell)\to L^2(\XX,\ell)$ be a symmetric and Hilbert-Schmidt
  integral operator with kernel $K(\cdot,\cdot)$. Then, there exists a complete
  orthonormal basis $(h_n)_{n\in\bbbn}\subset L^2(\XX,\ell)$ and a decreasing
  sequence $(\lambda_n)_{n\in\bbbn}$ converging to $0$ such that, for any
  $x\in\XX$,

\begin{equation*}
  Kf(x)=\displaystyle\sum_{n=1}^{+\infty}\lambda_n \langle f,h_n \rangle h_n(x),
\end{equation*}

or equivalently, for any $x,y\in\XX$,

\begin{equation*}
  K(x,y)=\displaystyle\sum_{n=1}^{+\infty} \lambda_n h_n(x) h_n(y).
\end{equation*}

Since $K$ is Hilbert-Schmidt, $$\displaystyle\sum_{n=1}^{+\infty}
\lambda_n^2<+\infty,$$

and, if $K$ is trace-class, then $$\displaystyle\sum_{n=1}^{+\infty}
\lambda_n<+\infty.$$

Moreover, if for any $n\in\bbbn$, $1+\alpha\lambda_n\neq0$, then the integral
operator $J=(I+\alpha K)^{-1}K$ has a kernel $J$ given for any $x,y\in\XX$ by:

\begin{equation*}
  J(x,y)=\displaystyle\sum_{n=1}^{+\infty}\dfrac{\lambda_n}{1+\alpha\lambda_n}h_n(x)h_n(y).
\end{equation*}

\end{theo}

Focus now on transformations on $\alpha$-DPPs and fix an $\alpha$-DPP
$\Phi$ with kernel $K$. The reduction of $\Phi$ to a compact subset $\Lambda$ of
$\XX$ is also an $\alpha$-DPP whose kernel is given, for any $x,y\in\XX$, by:

\begin{equation}\label{reductiondeterminantal}
  K_\Lambda(x,y)=K(x,y)\textbf{1}_{\Lambda\times\Lambda}(x,y).
\end{equation}


If $\beta$ is a real number of $[0,1]$, then the $\beta$-thinning of $\Phi$ is
also an $\alpha$-DPP whose kernel is given, for any $x,y\in\XX$, by:

\begin{equation}\label{thinningdeterminantal}
  K_\beta(x,y)=\beta K(x,y).
\end{equation}


Supposing that $\XX=\bbbr^d$ and that $\epsilon$ is a positive real number, the
$\epsilon$-rescaling of $\Phi$ is also an $\alpha$-DPP whose kernel is given,
for any $x,y\in\bbbr^d$, by:

\begin{equation}\label{rescalingdeterminantal}
  K_\epsilon(x,y)=\dfrac{1}{\epsilon} K(\epsilon^{-\frac{1}{d}}x,\epsilon^{-\frac{1}{d}}y).
\end{equation}


Another important result concerns independent superpositions: For any
$n\in\bbbn$, a $(-1/n)$-determinantal point process with kernel $K$ is the
independent superposition of $n$ determinantal point processes with kernel
${n}^{-1}K$. Moreover, such a sequence of point processes converges in law to
a Poisson point process with intensity measure $K(x,x)\ell(\d x)$. All these
properties on transformations of $\alpha$-DPP may easily be verified by
calculations on the  Laplace transforms.


A key example of transformations of determinantal point processes is also given
by a parametric family of determinantal point processes on the complex space:
the class of $\beta$-Ginibre point processes, where the parameter $\beta$ is a
real number between $0$ and $1$.

The {Ginibre point process} (GPP) with intensity $\rho=\frac{\gamma}{\pi}$ (with
$\gamma>0$) is a determinantal point process on $\mathbb{C}$ whose kernel
$K_\gamma$ is given for any $x,y\in\mathbb{C}$ by:

\begin{equation*}
  K_\gamma(x,y)=\dfrac{\gamma}{\pi}e^{-\frac{\gamma}{2}(|x|^2+|y|^2)}e^{\gamma x\overline{y}}.
\end{equation*}

It was introduced in 1965 by Ginibre \cite{ginibre_statistical_1965} and may be
interpreted as a point process with Gaussian repulsion between its particles.
If $\beta$ is a real number between $0$ and $1$, the {$\beta$-Ginibre point
  process} ($\beta$-GPP) with intensity $\rho=\frac{\gamma}{\pi}$ is a
determinantal point process on $\mathbb{C}$ whose kernel $K_{\gamma,\beta}$ is
given for any~$x,y\in\mathbb{C}$ by:

\begin{equation*}
  K_{\gamma,\beta}(x,y)=\dfrac{\gamma}{\pi}e^{-\frac{\gamma}{2\beta}(|x|^2+|y|^2)}e^{\frac{\gamma}{\beta}x\overline{y}}.
\end{equation*}


A $\beta$-Ginibre point process may be built by combining two operations on a
Ginibre point process: a thinning with parameter $\beta$ (one keeps each point
independently with probability~$\beta$) then a rescaling with parameter
$\sqrt{\beta}$, such that we keep the same intensity. Hence, the parameter
$\beta$ provides an information concerning the degree of repulsiveness of the
point process: the smaller $\beta$ is, the less repulsive the $\beta$-Ginibre
point process is. Note that such a point process is not defined for $\beta>1$.

When $\beta$ tends to $0$, a $\beta$-Ginibre point process is close to a Poisson
point process. In this sense, a Poisson point process may be considered as a
$\beta$-Ginibre point process with $\beta=0$. The simulation of $\beta$-Ginibre
point processes is investigated in \cite{decreusefond_note_2015}. Among results
on this topic, see the article of Goldman about its Palm measure and
Voronoi tessellation \cite{goldman_palm_2010}, its link with random matrices
\cite{liu_universality_2016} and some applications to wireless networks
\cite{deng_ginibre_2014, gomez_case_2015}.


\begin{figure}[!h]
  \centering \includegraphics*[width=0.3\textwidth]{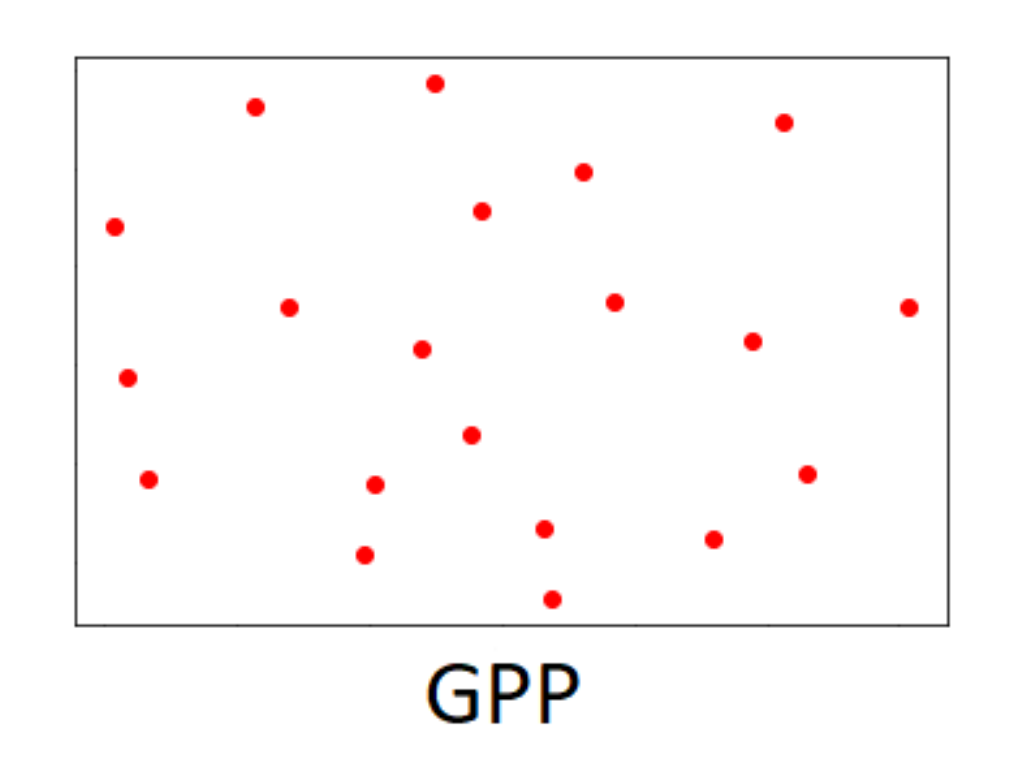}
  \includegraphics*[width=0.3\textwidth]{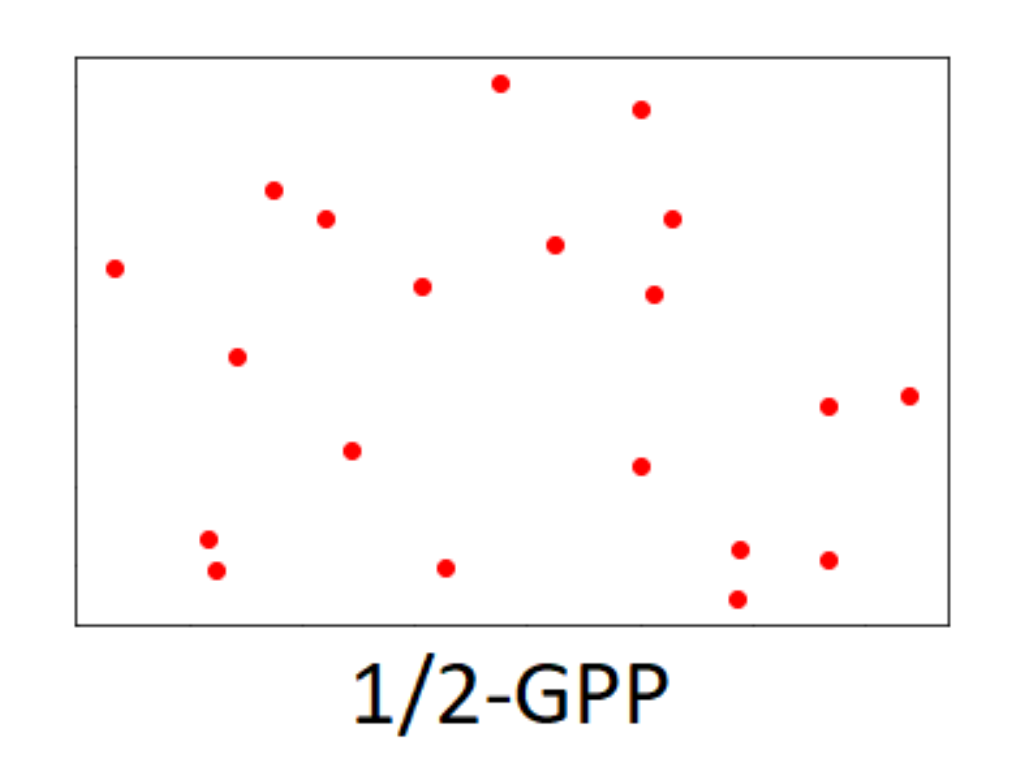}
  \includegraphics*[width=0.3\textwidth]{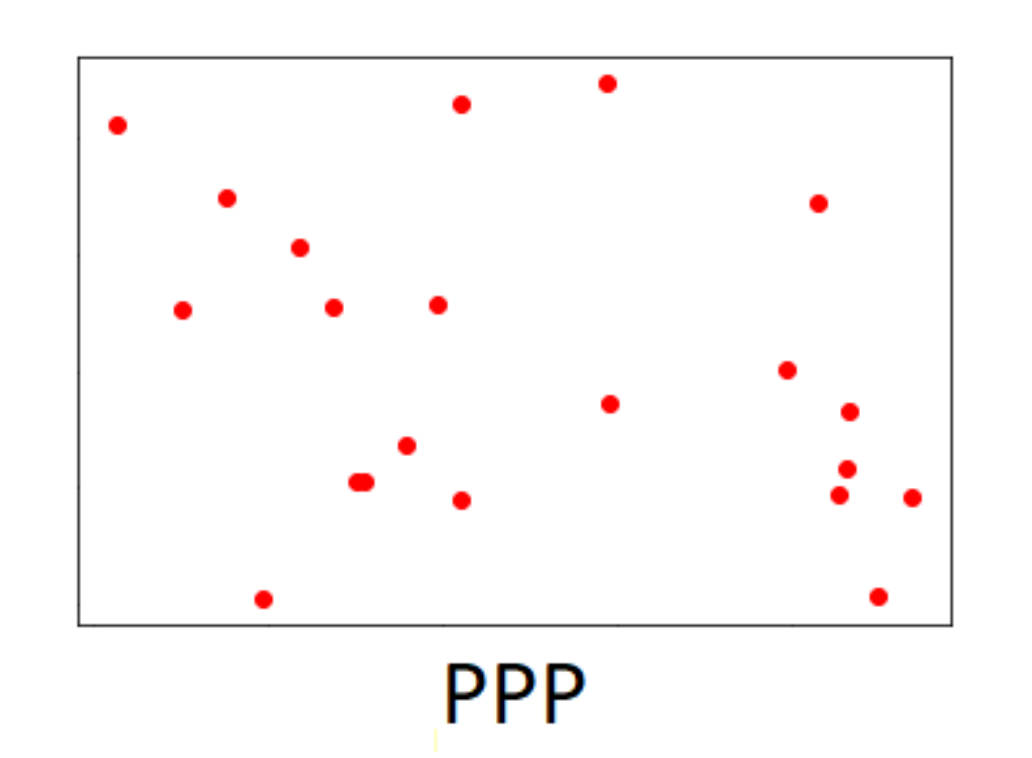}
	\caption{\label{betaginibre} Realizations of a Ginibre point process, a
    ($1/2$)-Ginibre point process and a Poisson point process.}
  \hfill
\end{figure}

Some realizations of a Ginibre point process, a ($1/2$)-Ginibre point process
and a Poisson point process are given in Figure~\ref{betaginibre}.
One can observe that the repulsiveness between the particles is indeed  weaker and
weaker as $\beta$ decreases and almost null as $\beta$ tend to~$0$.

\subsection{Convergence} \label{subsec_convergence}

The asymptotic evolution of a family of point processes can be viewed through
the prism of various topologies. The following definitions provides classical
notions of convergence on the space of point processes. A function $h:\bbbl\to\bbbr$ is said to be {$1$-Lipschitz}
according to a distance $\mmmd$ on a subset $\bbbl$ of $\bbbm$ if for any
$\nu_1,\nu_2\in\bbbl$,
$$
|h(\nu_1)-h(\nu_2)|\leq \mmmd(\nu_1,\nu_2),
$$
and $\Lip_1(\bbbl,\mmmd)$ denotes the set of all these maps which are
measurable. It may also be useful to recall here that a sequence
$(M_n)_{n\in\bbbn}\subset\bbbm$ {converges vaguely} to a measure $M\in\bbbm$ if,
for any positive function $f\in\Ccal_K(\XX,\bbbr)$,

\begin{equation*}
  \lim_{n\to+\infty}\int_\XX f(x)M_n(\d x)=\int_\XX f(x)M(\d x),
\end{equation*}

and that the space $\bbbm$ endowed with the vague convergence is a Polish space.
This enables to then define the most common type of convergence used on point
processes, which is convergence in law, that is, convergence for the Laplace
functionals. Formally, a sequence $(\Phi_n)_{n\in\bbbn}$ of point processes on
$\XX$ {converges in law} to a point process $\Phi$ on $\XX$ if, for any bounded
and continuous for the vague topology function $F:\N_\XX\to\bbbr$,

\begin{equation*}
  \lim_{n\to+\infty}\Esp[F(\Phi_n)]=\Esp [F(\Phi)].
\end{equation*}

Such a topology is metrizable and it is shown \cite{kallenberg_random_1983} that
the distance $\mmmd_P$, called here Polish distance, between two measures
$\nu_1$ and $\nu_2$ given by

$$\mmmd_P(\nu_1,\nu_2)=\displaystyle\sum_{k=1}^{+\infty}\dfrac{1}{2^k}\dfrac{|\langle f_k,\nu_1 \rangle - \langle f_k,\nu_2 \rangle|}{1+|\langle f_k,\nu_1 \rangle - \langle f_k,\nu_2 \rangle|},$$

where $f=(f_k)_{k\in\bbbn}$ is a sequence of functions generating $\Xcal$ and
where for any $x\in\bbbr_+$,

$$\Psi(x)=\dfrac{x}{1+x},$$
defines a metric for this topology and then provides a way to precise some
convergence rates. For this last definition, we can assume without loss of
generality that the Polish distance on $\bbbm'$ is chosen such that
$f\subset\Ccal_K(\bbbm)\cap\Lip_1(\mmmd_{TV})$.


A central question in the field of optimal transport is to determine for which
coupling between two random elements $X_1$ and $X_2$ the value of a given cost
function $\Delta$ is minimal. If $X_1$ and $X_2$ are some point processes and
the cost function $\Delta$ is a distance on the configuration space, the optimal
transport cost $\mmmd^*$, also called Kantorovich-Rubinstein distance, between
the probability distributions $\P_1$ and $\P_2$ of $X_1$ and $X_2$ is defined as

\begin{equation} \label{Kantodistance}
  \mmmd^*(\P_1,\P_2):=\inf_{\bC\in\Sigma(\P_1,\P_2)}\int_{\N_\XX\times\N_\XX}\mmmd(\omega_1,\omega_2)\bC(\d(\omega_1,\omega_2))
\end{equation}

where $\Sigma(\P_1,\P_2)$ denotes the set of probability measures on
$\bbbm\times\bbbm$ with first marginal $\P_1$ and second marginal $\P_2$. In
this case, there is at least one coupling $\bC\in\Sigma(\P_1,\P_2)$ for which
the infimum is attained \cite{villani_optimal_2008}. Moreover, if $\P_1$ and
$\P_2$ are concentrated on $\N_\XX$, then there is at least one coupling
$\bC\in\Sigma(\P_1,\P_2)$ for which the infimum is attained, and the Kantorovich
duality theorem (see e.g. \cite{villani_optimal_2008}) says that this minimum
equals

\begin{equation*}
  \mmmd^*(\P_1,\P_2)=\sup \Big|\int_{\N_\XX} F(\omega)  \P_1(\d\omega)  -\int_{\N_\XX} F(\omega)  \P_2(\d\omega) \Big|,
\end{equation*}

where the supremum is over all $F\in\Lip_1(N_\XX,\mmmd)$ that are integrable
with respect to $\P_1$ and ~$\P_2$. This distance associated to the total
variation distance provides a strong topology on point processes since it is
strictly finer than for the total variation distance
\cite{decreusefond_functional_2016}, and it is shown by Decreusefond et al. that
this Kantorovich-Rubinstein distance between finite Poisson point processes is
bounded by the total variation distance between its intensity measures
\cite{decreusefond_upper_2010}. It is also included in the larger class of
so-called Wasserstein distances, where a $L^p$-distance appears in
(\ref{Kantodistance}). Wasserstein distances were investigated in
\cite{barbour_steins_2006, decreusefond_wasserstein_2008, last_normal_2016} for
Poisson point processes, and more recently in \cite{del_moral_stability_2017} by
Del Moral and Tugaut for Kalman-Bucy filters, and in the field of persistent
homology in \cite{chazal_subsampling_2015} where Chazal et al. compare this
distances with so-called bottleneck, Hausdorff and Gromov-Hausdorff distances.
Other examples of metrics on point processes are proposed in
\cite{schulte_distances_2014} and \cite{schuhmacher_new_2008}.


In this work, the cost function $\mmmd$ denotes alternatively the {discrete
  distance} $\mmmd_D$ on $\bbbm$ defined for any $\nu_1,\nu_2\in\bbbm$ by
$$\mmmd_D(\nu_1,\nu_2):=\textbf{1}_{\{\nu_1\neq\nu_2\}}$$
or the {total variation distance} $\mmmd_{TV}$ on $\bbbm$ defined for any
$\nu_1,\nu_2\in\bbbm$ by
$$\mmmd_{TV}(\nu_1,\nu_2):=\sup_{\substack{{A\in\Xcal} \\ {\nu_1(A),\nu_2(A)<\infty}}}|\nu_1(A)-\nu_2(A)|.$$


Note that, for any $\nu_1,\nu_2\in\bbbm$,

\begin{equation*} \label{remdTV} \mmmd_{TV}(\nu_1,\nu_2)=\int_\XX\Big|\dfrac{\d
    \nu_1(x)}{\d(\nu_1+\nu_2)}-\dfrac{\d
    \nu_2(x)}{\d(\nu_1+\nu_2)}\Big|(\nu_1+\nu_2)(\d x)
\end{equation*}
and, in particular, if $\nu_1,\nu_2\in \widehat{N}_\XX$, then

\begin{equation*}
  \mmmd_{TV}(\nu_1,\nu_2)=|\nu_1\setminus\nu_2|+|\nu_2\setminus\nu_1|.
\end{equation*}

The existing links between the previous topologies may be summarized as follows.
For any distance $\mmmd$ on the space of point processes on $\XX$,
$\Tcal(\mmmd)$ denotes the topology associated to the corresponding metric
space. Then,

\begin{itemize}

\item ${\mmmd_P}_{|\bbbm_1'}$ and ${\mmmd_P}_{|N_\XX}^*$ both provide a metric
  for the convergence in law \cite{kallenberg_random_1983};

\item
  $\Tcal({\mmmd_P}_{|\bbbm_1'})=\Tcal({\mmmd_P}_{|N_\XX}^*)\subsetneq\Tcal({\mmmd_{TV}}_{|\bbbm_1'})\subsetneq\Tcal({\mmmd_{TV}}_{|N_\XX}^*)$
  \cite{decreusefond_functional_2016};

\item ${\mmmd_{TV}}_{|\bbbm_1'}={\mmmd_D}_{|N_\XX}^*\leq{\mmmd_{TV}}_{|N_\XX}^*$
  \cite{decreusefond_asymptotics_2017};

\item ${\mmmd_P}_{|\bbbm_1'}\leq{\mmmd_{TV}}_{|N_\XX}^*$.

\end{itemize}

\section{Stein's method for a finite Poisson process} \label{sec_steins_method}


This section aims to present Stein's method applied to a finite Poisson point
process. We use the Stein's principle and the construction given in
\cite{decreusefond_functional_2016}, but our proofs are sometimes different,
highlighting the properties of the thinning operation and the invariance of the
Poisson process distribution: for any Poisson point process $\Phi$ and
any~\mbox{$t\in[0,1]$},

\begin{equation*}
  t\circ\Phi^{(1)}+(1-t)\circ\Phi^{(2)}\overset{\Dcal}{=}\Phi,
\end{equation*}

where $\Phi^{(1)}$ and $\Phi^{(2)}$ are independent copies of $\Phi$.



For any $t\in\bbbr_+$, let $P_t$ be an operator on the space of measurable and
bounded functions $F:\widehat{N}_\XX\to\bbbr$. We say that the family
$(P_t)_{t\geq0}$ is a {semi-group} on $\widehat{N}_\XX$ if, for any
$t,s\in\bbbr_+$,

\begin{equation*}
  P_{t+s}=P_t\circ P_s,
\end{equation*}
and its {infinitesimal generator} $L$ is defined for any measurable and bounded
function $F:\widehat{N}_\XX\to\bbbr$ and any $\phi\in \widehat{N}_\XX$ such that
$t\mapsto P_tF(\phi)$ is differentiable in $0$ by:

\begin{equation*}
  LF(\phi)=\left.\frac{\d P_tF(\phi)}{\d t}\right|_{t=0}.
\end{equation*}


We focus now on a Poisson point process $\zeta_M$ with finite intensity measure
$M$, and associate to $\zeta_M$ its {Glauber process} $(G_t)_{t\geq0}$, defined
for any $t\in\bbbr_+$ and $\phi\in\widehat{N}_\XX$ by

\begin{equation*}
  G_t(\phi)=e^{-t}\circ\phi+(1-e^{-t})\circ\zeta_M.
\end{equation*}

For any $t\in\bbbr_+$, we consider the operator $P_t$ is defined for any
measurable and bounded function $F:\widehat{N}_\XX\to\bbbr$ and any $\phi\in
\widehat{N}_\XX$ by:

\begin{equation*}
  P_tF(\phi)=\Esp[F(G_t(\phi))]=\Esp[F(e^{-t}\circ\phi+(1-e^{-t})\circ\zeta)].
\end{equation*}


Its dynamics can be described as follows: imagine a homogeneous Poisson process
$\zeta_b$ on $\bbbr_+$ with intensity $M(\XX)$. The jump times of $\zeta_b$
determine the birth times of the particles in $\zeta$, placed in $\XX$ according
to the distribution $\frac{M(\cdot)}{M(\XX)}$. The lifetime of each particle is
exponentially distributed with parameter $1$.


\begin{theo}\label{bienunsemigroup}
  The family $(P_t)_{t\geq0}$ is a semi-group.
\end{theo}
Note that, in the previous proof, we only use associativity of thinning and the
invariance property of a Poisson point process distribution.
According to the Stein's approach, this Glauber process verifies as expected the
following properties: on one hand it is invariant for the Poisson point process
$\zeta$, on the other hand it is ergodic, which is summarized in the following
theorem.

\begin{lem}\label{ergodic}

  For any $t\geq0$,

\begin{equation*}
  G_t(\zeta)\stackrel{\Dcal}{=}\zeta.
\end{equation*}

For any $F\in\Lip_1(\widehat{N}_\XX,\mmmd_{TV})$ and any $\phi\in
\widehat{N}_\XX$,

\begin{equation*}
  \displaystyle\lim_{t\to+\infty}P_tF(\phi)=\Esp[F(\zeta)].
\end{equation*}

\end{lem}


In order to analyze the variations of any given functional of a configuration or
a point process, we need to introduce the following operator. The {gradient}
$D_x$ in direction $x\in\XX$ is defined, for any measurable function
$F:N_\XX\to\bbbr$ and any $\phi\in N_\XX$, by

\begin{equation*}
  D_x F(\phi)=F(\phi+x)-F(\phi).
\end{equation*}

Its association with the Poisson point process $\zeta$ and its semi-group
presents some relevant characteristics. One of its main advantages is the
closability property provided by the following theorem.

\begin{theo}\label{gradientprop0}

  If $F,G:N_\XX\to\bbbr$ are two measurable and bounded functions such that
  $F(\phi)=G(\phi)$ $\P_{\zeta}(\d\phi)$-a.s., then

\begin{equation*}
  D_x F(\phi)=D_x G(\phi)\ (M\otimes\P_{\zeta})(\d x,\d\phi)\text{-a.s.}.
\end{equation*}

\end{theo}

The gradient, thus defined, also appears in the expression of the infinitesimal
generator associated to the semi-group $(P_t)_{t\geq0}$, given in the next
theorem.

\begin{theo}\label{generator}

  The infinitesimal generator $L$ associated to $(P_t)_{t\geq0}$ is given for
  any measurable and bounded function $F:\widehat{N}_\XX\to\bbbr$ and any
  $\phi\in \widehat{N}_\XX$ by

\begin{equation*}
  LF(\phi)=\int_{\XX} D_x F(\phi)M(\d x)+\displaystyle\sum_{y\in\phi} (F(\phi\setminus y)-F(\phi)).
\end{equation*}

Moreover, a point process $\Phi$ is a Poisson point process with intensity
measure $M$ if and only if, for any measurable and bounded function
$F:\widehat{N}_\XX\to\bbbr$,

\begin{equation*}
  \Esp[LF(\Phi)]=0.
\end{equation*}

\end{theo}

Another essential property of the gradient, still when it is coupled with the
semi-group $(P_t)_{t\geq0}$, is the following commutation relation.

\begin{lem}\label{semigroupgradientcom}

  For any $t\in\bbbr_+$, any $x\in\XX$, any measurable and bounded function
  $F:\widehat{N}_\XX\to\bbbr$ and any $\phi\in \widehat{N}_\XX$,

\begin{equation*}
  D_x P_tF(\phi)=e^{-t}P_tD_{x}F(\phi).
\end{equation*}

\end{lem}


This approach using a generator provides a solution to the Stein's equation in
the following result, the so-called Stein-Dirichlet representation formula. This
theorem is only based on the definitions of semi-group and infinitesimal
generator, together with the ergodic property.

\begin{theo}\label{SDR}

  For any $F\in\Lip_1(\widehat{N}_\XX,\mmmd_{TV})$ and any $\phi\in
  \widehat{N}_\XX$,

\begin{equation*}
  \Esp[F(\zeta)]-F(\phi)=\int_0^{+\infty} LP_sF(\phi)\d s.
\end{equation*}

\end{theo}

In \cite{barbour_steins_1992}, Barbour and Brown apply Stein's method to Poisson
point process. They deduce an upper bound for the total variation distance
between a finite Poisson point process and another finite point process using
Palm measure, given by the following inequality: If $\Phi$ is a point process of
intensity $M$,
\begin{equation*}
  \mmmd_{TV}^*(\Phi,\zeta_{M})\leq\int_{\XX}\Esp[\|\Phi- \Phi_{x}\backslash x\|_{\text{TV}}]M(\d x)
\end{equation*}
where $\Phi_{x}$ is a point process constructed on the same probability as
$\Phi$, which has the distribution of the Palm measure of $\Phi$ at $x$.

Our approach, which leads to the following fundamental theorem, focuses on
Papangelou intensity rather than Palm measure - a functional rather than a
probability measure - in order to give an easier way to perform calculations.

\begin{theo}\label{Steinprop} 

  If $\zeta_{M}$ is a Poisson point process on $\XX$ with finite intensity measure
  $M(\d x)=m(x)\ell(\d x)$ and $\Phi$ is a second finite point process on $\XX$
  with Papangelou intensity $c$, then

\begin{equation*}
  \mmmd_{TV}^*(\Phi,\zeta_{M})\leq \int_\XX \Esp[|m(x)-c(x,\Phi)|] \ell(\d x).
\end{equation*}

\end{theo}

The proof of this fundamental result synthesizes the main results which are
exposed previously: from the Stein-Dirichlet representation formula, it becomes
possible to apply successively the expression of the generator $L$ and the
definition of the Papangelou intensity. The expected upperbound is then obtained
using the outstanding properties of the gradient.

\section{Papangelou intensity and repulsiveness} \label{sec_papangelou}


Following Georgii and Yoo \cite{georgii_conditional_2005}, we define
repulsiveness according to the Papangelou intensity: a point process $\Phi$ on
$\XX$ with a version $c$ of its Papangelou intensity is said to be {repulsive}
(according to $c$) if, for any $\omega,\phi\in\N_\XX$ such that
$\omega\subset\phi$ and any $x\in\XX$,

\begin{equation} \label{repulsive} c(x,\phi)\leq c(x,\omega).
\end{equation}

Such a definition is in particular verified for Gibbs and determinantal point
processes, which belong to the main categories of point processes used in order
to model repulsiveness. The relevance of this definition also appears when
considering the intuitive approach of the Papangelou intensity: the probability
of finding a particle in the location $x\in\XX$ given the configuration $\phi$
included in $\Phi$ is as weak as the number of particles in $\phi$ is high.
However, this definition may be considered as quite restrictive and it seems to
be sufficient to consider point processes with larger assumptions: in our
approach, a point process is defined as {weakly repulsive} (according to its
Papangelou intensity $c$) if, for any $\phi\in\N_\XX$ and any $x\in\XX$,

\begin{equation*}
  c(x,\phi)\leq c(x,\varnothing).
\end{equation*}


This last definition presents several advantages: it includes repulsive point
processes, the required inequality is obtained with easier computations and a
weakly repulsive point process verifies some useful properties. In particular,
the following lemmas though very elementary are the key to the next results.

\begin{lem}\label{repulsivelem1}

  If $\Phi$ is a finite and weakly repulsive point process on $\XX$ with
  Papangelou intensity $c$ and void probability $p_0$, then for any $x\in\XX$,

\begin{equation*}
  |c(x,\varnothing)-\rho(x)|\leq (1-p_0)c(x,\varnothing).
\end{equation*}

\end{lem}

\begin{lem}\label{repulsivelem2}

  If $\Phi$ is a finite and weakly repulsive point process on $\XX$ with
  Papangelou intensity $c$, then for any $x\in\XX$,
\begin{equation*}
  \Esp[|c(x,\Phi)-\rho(x)|]\leq 2\Bigl( (c(x,\varnothing)-\rho(x) \Bigr).
\end{equation*}
\end{lem}
These two  results are shown by directly using the previous definition of weak
repulsiveness and highlight the prominent role of the quantity
$c(x,\varnothing)$ where $x\in\XX$, which leads to consider the following
property.

\begin{lem}\label{Papangeloulem1}

  Let $\Phi$ be a finite point process on $\XX$ with Papangelou intensity $c$.
  Then,
$$\P(|\Phi|=1)=\P(|\Phi|=0)\int_\XX c(x,\varnothing)\d x.$$
\end{lem}
We now show how to compute the Papangelou intensity for different
transformations of point processes.
\begin{theo}\label{papreduction}
  Let $\Phi$ be a point process on $\XX$ with Papangelou intensity $c$,
  $\Lambda$ a compact subset of $\XX$ and $\Phi_{|\Lambda}$ the reduction of
  $\Phi$ to $\Lambda$. Then, its Papangelou intensity $c_{\Lambda}$ verifies for
  any
  $x\in\XX$ $$c_{\Lambda}(x,\Phi_{|\Lambda})=c(x,\Phi)\textbf{1}_{\{x\in\Lambda\}}\
  \text{a.s.}.$$
\end{theo}
\begin{theo}\label{papsuperposition}
  Let $\Phi_1,\dots,\Phi_n$ ($n\in\bbbn$) be independent point processes on
  $\XX$ with respective Papangelou intensities $c_1,\dots,c_n$ and $\Phi$ their
  independent superposition. Then, its Papangelou intensity $c$ verifies for any
  $x\in\XX$
  \begin{equation}
    \label{eq_Article_Stein:2}
    c\Big(x,\displaystyle\sum_{i=1}^n\Phi_i\Big)=\displaystyle\sum_{i=1}^nc_{i}(x,\Phi_i)\ \text{a.s.}.
  \end{equation}
\end{theo}
The right hand side in \eqref{eq_Article_Stein:2} is not truly the
Papangelou intensity of $\sum_{i=1}^n\Phi_i$ since it is not
$(\sum_{i=1}^n\Phi_i)$-measurable but this ersatz is sufficient for our
purposes. A direct consequence of this result is the weak repulsiveness of an
independent superposition of weakly repulsiveness point processes, as provided
by the following corollary.

\begin{cor}\label{corpapsuperposition}

  Let $\Phi_1,\dots,\Phi_n$ ($n\in\bbbn$) be independent and weakly repulsive
  point processes on $\XX$. Then their independent superposition is also weakly
  repulsive.

\end{cor}

\begin{theo}\label{papthinning1}

  Let $\Phi$ be a point process on $\XX$, let $\beta$ be a function from $\XX$
  to $[0,1]$ and $\beta\circ\Phi$ the $\beta$-thinning of $\Phi$. Then, its
  Papangelou intensity $c_\beta$ verifies, for any $x\in\XX$,
$$c_\beta(x,\beta\circ\Phi)=\beta(x)\Esp[c(x,\Phi)\ |\ \beta\circ\Phi]\ \text{a.s.}.$$
\end{theo}
\begin{theo}\label{paprescaling}

  Let $\Phi$ be a point process on $\bbbr^d$ with Papangelou intensity $c$, let
  $\epsilon$ be a positive real number and $\Phi^{(\epsilon)}$ the
  $\epsilon$-rescaling of $\Phi$. Then, its Papangelou intensity
  $c^{(\epsilon)}$ is given for any $x\in\bbbr^d$ and
  $\phi\in\widehat{\N}_{\bbbr^d}$ by
$$c^{(\epsilon)}(x,\phi)=\dfrac{1}{\epsilon}c(\epsilon^{-\frac{1}{d}}x,\epsilon^{-\frac{1}{d}}\phi).$$
\end{theo}

We now give an expression of the Papangelou intensity for classical point
processes. First of all, the Papangelou intensity of a Poisson point process has
a very simple expression, given in the next theorem.

\begin{theo}\label{Poissonprop2bis}

  Let $\Phi$ be a Poisson point process with intensity measure~$M(\d x)=m(x)\d
  x$. Then, its Papangelou intensity $c$ is given for any $x\in\XX$ and
  $\phi\in\widehat{\N}_\XX$ by
$$c(x,\phi)=m(x).$$

\end{theo}

Note that this last result states in a formal way the independence property
of a Poisson point process: the probability of finding a particle in a given
location does not depend on the other particles of the configuration. It is a
direct consequence of the Mecke formula for a Poisson point process and is also
mentioned in \cite{georgii_conditional_2005}, but it is also possible to prove
it by using the expression of the Janossy functions and their link with
Papangelou intensities, given by the formula~(\ref{Papangelouprop1}). We use this
approach in the following theorems to get the expression of the Papangelou
intensity for some other Poisson-like point processes.

\begin{theo}\label{purelyrandomprop2bis}

  Let $\Phi$ be a purely random point process on $\XX$ supported by a
  distribution $(p_n)_{n\in\bbbn_0}$ such that $p_n\neq0$ for any $n\in\bbbn_0$,
  and a probability measure $\mu(\d x)=~q(x)\ell(\d x)$. Then its Papangelou
  intensity $c$ is given for any $n\in\bbbn_0$ and any $x,x_1,\dots,x_n\in\XX$
  by

$$c(x,\{x_1,\dots,x_n\})=(n+1)\dfrac{p_{n+1}}{p_n}q(x).$$

Moreover, $\Phi$ is repulsive if and only if, for any $n\in\bbbn_0$,
\begin{equation*}
  (n+1)p_{n+1}^2\geq(n+2)p_np_{n+2}
\end{equation*}
and weakly repulsive if and only if, for any $n\in\bbbn_0$,
\begin{equation*}
  p_0(n+1)p_{n+1}\leq p_np_1.
\end{equation*}
\end{theo}

\begin{theo}\label{CondPPPprop2bis}

  Let $\Phi$ be a conditional Poisson point process with intensity measure $M(\d
  x)=m(x)\d x$ and conditional set $C$. Then its Papangelou intensity $c$ is
  given for any $n\in\bbbn_0$ and any $x,x_1,\dots,x_n\in\XX$ by
$$c(x,\{x_1,\dots,x_n\})=m(x)\textbf{1}_{\{x_1,\dots,x_n,x\}\in C}\textbf{1}_{\{x_1,\dots,x_n\}\in C}.$$

Moreover, if $C$ is decreasing, then $\Phi$ is repulsive.

\end{theo}

In the same way, since we define a Gibbs point process according to its Janossy
function, its Papangelou intensity is also obtained immediately.

\begin{theo}\label{Gibbsprop1}

  Let $\Phi$ be a Gibbs point process with temperature parameter $\theta>0$ and
  total potential energy
$$U(x_1,\dots,x_n)=\displaystyle\sum_{r=1}^n\displaystyle\sum_{1\le i_1<\dots<i_r\le n} \Psi_r(x_{i_1},\dots,x_{i_r}),$$
where $\Psi_r$ is the $r^{\text{th}}$-order interaction potential. Then its
Papangelou intensity $c$ is given for any $x\in\XX$ and
$\phi\in\widehat{\N}_\XX$ by

$$c(x,\phi)=e^{-\theta (U(x\phi)-U(\phi))}.$$

Moreover, $\Phi$ is repulsive.

\end{theo}

The following result provides an explicit expression for the Papangelou
intensity of an $\alpha$-DPP. 

\begin{theo}\label{DPPpap}

  Let $\Phi$ be an $\alpha$-DPP with kernel $K$ and associated kernel $J$. Then
  its Papangelou intensity $c$ is given for any $x\in\XX$ and
  $\phi\in\widehat{\N}_\XX$ by
$$c(x,\phi)=\dfrac{{\det}_\alpha J(x\phi,x\phi)}{{\det}_\alpha J(\phi,\phi)}\cdotp$$
Moreover, if $\alpha=-1$, then $\Phi$ is repulsive.

\end{theo}

\section{Applications} \label{sec_applications}

In this section, we apply the upper bound given in Theorem \ref{Steinprop} to
provide an estimation of the distance between some classical point processes and
a Poisson or Cox point process. The residual computations essentially use the
results of Section \ref{sec_papangelou} on Papangelou intensities.

\subsection{Application to Poisson-like point
  processes}\label{subsec_applicationpoissonlike}


We focus in this subsection on Poisson-like point processes. In a sense, the
results presented here may be seen as some generalizations of the two following
results for respectively finite Poisson and Cox point processes, which have
already been shown in \cite{decreusefond_upper_2010}.

\begin{theo}\label{Superpoisson}

  Let $\zeta_1,\zeta_2$ be two Poisson point processes on $\XX$ with respective
  intensity measures $M_1$ and $M_2$. Then,
\begin{equation*}
  \mmmd_{TV}^*(\zeta_1,\zeta_2)\leq\mmmd_{TV}(M_1,M_2).
\end{equation*}

\end{theo}

\begin{theo}\label{Supercox}

  Let $\Gamma_1,\Gamma_2$ be two Cox point processes on $\XX$ directed by
  respective almost surely finite random measures $M_1$ and $M_2$. Then,
\begin{equation*}
  \mmmd_{TV}^*(\Gamma_1,\Gamma_2)\leq\mmmd_{TV}^*(M_1,M_2).
\end{equation*}

\end{theo}


The topology used in Theorem \ref{Supercox} may be too strong. In this case, it
is relevant to mention that a similar result can be obtained for
$\mmmd_{TV}(=\mmmd_D^*)$ instead of $\mmmd_{TV}^*$. Indeed, since trivially, for
any Poisson point processes $\zeta_1$ and $\zeta_2$ with respective intensity
measures $M_1$ and $M_2$,
\begin{equation*}
  \mmmd_{D}^*(\zeta_1,\zeta_2)\leq \mmmd_D(M_1,M_2),
\end{equation*}

it follows, by adapting the proof of Theorem \ref{Supercox}, that
\begin{equation*}
  \mmmd_{TV}(\Gamma_1,\Gamma_2)\leq\mmmd_D^*(M_1,M_2).
\end{equation*}


In the following theorem, the Poisson point process of Theorem
\ref{Superpoisson} is replaced by a purely random point process.

\begin{theo}\label{Supertheo5}

  Let $M$ be a finite measure on $\XX$ such that $M(\d x)=m(x)\d x$ and
  $\mu\in\bbbm_1$ such that $\mu(\d x)=\frac{m(x)}{M(\XX)}\d x$. Let $\Phi$ be a
  purely random point process on $\XX$ supported by $\mu$ and the distribution
  $(p_n)_{n\in\bbbn_0}$ such that $p_n\neq0$ for any $n\in\bbbn_0$. Then,

$$\mmmd_{TV}^*(\Phi,\zeta_M)\leq\displaystyle\sum_{n=0}^{+\infty}\big|(n+1)p_{n+1}-M(\XX)p_n\big|,$$

where $\zeta_M$ is the Poisson point process on $\XX$ with intensity measure
$M$.

\end{theo}

Note that in the previous inequality, the right hand side is null as soon as
$\Phi$ is a purely random point process supported by $\mu$ and a Poisson
distribution with parameter $M(\XX)$, that is, as expected, a Poisson point
processes with intensity measure $M$.

In the following theorem, we apply our upper bound to conditional Poisson point
processes.


\begin{theo}\label{Supertheo7}

  Let $\Phi$ be a Poisson point process with finite intensity measure $M(\d
  x)=m(x)\d x$. Let $\Phi_C$ be the conditional Poisson point process associated
  to $\Phi$ with intensity measure $M$ and condition $C\in\Ncal_\XX$. Then,
\begin{equation*}
  \mmmd_{TV}^*(\Phi_C,\Phi)\leq\int_\XX m(x)\bbbp(\Phi_C x\notin C)\d x.
\end{equation*}

\end{theo}

We may notice that the precision of the estimation is directly provided by the
weakness of the condition $C$. In particular, for $C=N_\XX$, $\Phi_C$ is a
Poisson point process with intensity measure $M$ and the bound is null. This
last result on conditional Poisson point processes may also be observed in the
more particular cases of hardcore and bounded Poisson point processes, which are
respectively given in the two following corollaries.

\begin{cor}\label{cor1Supertheo7}

  Let $\Phi$ be a Poisson point process with finite intensity $\lambda$
  restricted to a relatively compact subset $\Lambda$ of $\XX=\bbbr^d$. Let
  $\Phi_{R}$ be the hardcore Poisson point process associated to $\Phi$ with
  parameter measure $M$ and parameter $R>0$. Then,
\begin{equation*}
  \mmmd_{TV}^*(\Phi_{R},\Phi)\leq \dfrac{\lambda^2|\Lambda|}{p_{R}}V_d(R)
\end{equation*}

where
\begin{equation*}
  p_R=\bbbp(\forall x,y\in\Phi,x\neq y\implies \mmmd_{\XX}(x,y)\geq R)
\end{equation*}

and
\begin{equation*}
  V_d(R)=\dfrac{\pi^{\frac{d}{2}}R^d}{\Gamma(\frac{d}{2})}.
\end{equation*}

\end{cor}

\begin{cor}\label{cor2Supertheo7}

  Let $\Phi$ be a Poisson point process with finite intensity measure $M(\d
  x)=m(x)\d x$. Let $\Phi_{N}$ be the bounded Poisson point process associated
  to $\Phi$ with parameter measure $M$ and parameter $N\in\bbbn_0$. Then,
\begin{equation*}
  \mmmd_{TV}^*(\Phi_{N},\Phi)\leq\dfrac{e^{-M(\XX)}}{p_{N}}\dfrac{(M(\XX))^{N+1}}{N!}
\end{equation*}

where $p_N=\bbbp(\Phi(\XX)\leq N)$.

\end{cor}

\subsection{Application to weakly repulsive point
  processes}\label{subsec_applicationrepulsive}


In this subsection, we apply Theorem \ref{Steinprop} to some transformations of
weakly repulsive point processes. The two main underlying intuitions are the
following: first, basic operations on point processes such that independent
superposition or thinning weaken the interactions between the particles therein
and then, in a way, inject some independence in the point process. It seems thus
possible to build with these operations a sequence tending to a point process
without interdependence between its particles, that is, a Poisson point
process. Second, it appears that point processes with repulsion provide a more
favorable ground for the construction of such a sequence: a point process with
some clusters might for instance have too strong and too numerous interactions
between the particles of a given cluster, and the transformations described
above are probably not sufficient to delete in a significant way the links of
dependence in this cluster, which leads us to restrict our analyze to weakly
repulsive point process.

In the following theorem, we consider a superposition of finite weakly repulsive
point processes.

\begin{theo}\label{supertheo}

  For any $n\in\bbbn$, let $\Phi_n$ the superposition of $n$ independent, finite
  and weakly repulsive point processes $\Phi_{n,1},\dots,\Phi_{n,n}$, with
  respective correlation functions $\rho_{n,1},\dots,\rho_{n,n}$ and let
  $\zeta_M$ be a Poisson point process with intensity measure $M(\d
  x)=m(x)\ell(\d x)$. Then,
\begin{equation*}
  \mmmd_{TV}^*(\Phi_n,\zeta_M)\leq R_n+2n\Big(\displaystyle\max_{i\in\{1,\dots,n\}}\int_{\XX}\rho_{n,i}(x)\ell(\d x)\Big)^2,
\end{equation*}

where
\begin{equation*}
  R_n:=\int_\XX \Big|\displaystyle\sum_{i=1}^n \rho_{n,i}(x)-m(x)\Big|\ell(\d x).
\end{equation*}

\end{theo}

\begin{rem}\label{supertheocor1}

  Under the assumptions and notations of Theorem \ref{supertheo}, and supposing
  moreover that there exists a real constant $C$ such that for any $n\in\bbbn$,
\begin{equation*}
  \displaystyle\max_{i\in\{1,\dots,n\}}\int_{\XX}\rho_{n,i}(x)\ell(\d x)\leq\dfrac{C}{n},
\end{equation*}

one has for any $n\in\bbbn$,
\begin{equation*}
  \mmmd_{TV}^*(\Phi_n,\zeta_M)\leq R_n+\dfrac{2C^2}{n}.
\end{equation*}

\end{rem}

Noting that a $(-1/n)$-determinantal point process may be written as a
superposition of determinantal point processes, the following corollary is a
direct consequence of Theorem \ref{supertheo}.

\begin{cor}\label{supertheocor2}

  Let $n\in\bbbn$, $\Phi_n$ be a finite $(-1/n)$-determinantal point process
  with kernel $K$ and $\zeta$ be a Poisson point process with intensity measure
  $K(x,x)\d x$. Then,
\begin{equation*}
  \mmmd_{TV}^*(\Phi_n,\zeta)\leq\dfrac{2}{n}\Big(\int_\XX K(x,x)\d x\Big)^2.
\end{equation*}

\end{cor}

In the next corollary, we present another consequence of Theorem
\ref{supertheo}, which appears when drawing independent and identically
distributed (i.i.d.) points on the real half-axis.

\begin{cor}\label{supertheocor3}

  Let $h$ be a probability density function on $[0,1]$ such that
  $h(0_+):=\lim_{x\to0_+}h(x)\in\bbbr$, and $\Lambda$ be a compact subset of
  $\bbbr_+$. For any $n\in\bbbn$, assuming that $X_{n,1},\dots,X_{n,n}$ are $n$
  i.i.d. random variables with probability density function
  $h_n=\frac{1}{n}h(\frac{1}{n} \cdot)$, the point process $\Phi_n$ defined as
  $\Phi_n=\{X_{n,1},\dots,X_{n,n}\}\cap\Lambda$ verifies the following
  inequality:

$$\mmmd_{TV}^*(\Phi_n,\zeta)\leq \int_\Lambda\Big|h\Big(\frac{1}{n}x\Big)-h(0_+)\Big|\d x + \dfrac{2}{n} \Big(\int_\Lambda h\Big(\frac{1}{n}x\Big) \d x\Big)^2 $$

where $\zeta$ is the homogeneous Poisson point process with intensity $h(0_+)$
reduced to $\Lambda$.

\end{cor}


By combining an independent superposition with an independent thinning, it
becomes possible to enlarge the assumptions on the point processes of the
superposition: in the following theorem, we replace the weak repulsiveness by a
condition on the variance of their Papangelou intensities.

\begin{theo}\label{supertheo1bis}

  Let $\Phi$ be a point process on a compact subset $\Lambda$ of $\XX$ with
  Papangelou intensity $c$ and intensity measure $M(\d x)=m(x)\d x$. Let $\zeta$
  be a Poisson point process with intensity measure $M$. For any $n\in\bbbn$,
  the point process $\Phi_n$ is defined by:
\begin{equation*}
  \Phi_n=\displaystyle\sum_{k=1}^n\dfrac{1}{n}\circ\Phi^{(k)},
\end{equation*}

where $\Phi^{(1)},\dots,\Phi^{(n)}$ are $n$ independent copies of $\Phi$. If
there exists an integrable function $K:\Lambda\to\bbbr_+$ such that, for any
$x\in\Lambda$, $\Var[c(x,\Phi)]\leq K(x)$, then
\begin{equation*}
  \mmmd_{TV}^*(\Phi_n,\zeta)\leq\dfrac{1}{\sqrt{n}}\int_\Lambda \sqrt{K(x)} \d x.
\end{equation*}

\end{theo}


In order to fully observe the ability of the independent thinning to weaken the
interactions between some particles of a point process, it is also interesting
to combine it with a rescaling, whose role is to compensate the loss of
intensity caused by the thinning operation. We apply this in the following
theorem for finite stationary determinantal point processes.

\begin{theo}\label{Supertheo3}

  Let $K$ be the kernel of a stationary determinantal point process $\Phi$ on
  $\bbbr^d$ with intensity $\lambda\in\bbbr$, $\Lambda$ be a compact subset of
  $\bbbr^d$, $\beta\in(0,1)$ and $\zeta_{\Lambda,\lambda}$ denotes the
  homogeneous Poisson point process with intensity $\lambda$ reduced to
  $\Lambda$. Let $\Phi_{\Lambda,\beta}$ be the point process on $\bbbr^d$
  obtained by combining a $\beta$-thinning with a $\beta$-rescaling on the point
  process $\Phi$ that one reduces to $\Lambda$. Then,
\begin{equation*}
  \mmmd_{TV}^*(\Phi_{\Lambda,\beta},\zeta_{\Lambda,\lambda})\leq\dfrac{2\beta}{1-\beta} \lambda|\Lambda|.
\end{equation*}

\end{theo}


The application to Gibbs point processes given in the following only focuses on
pairwise Gibbs point processes, that is such that, for any
$r\in\bbbn\setminus\{1,2\}$, $\Psi_r\equiv0$. The functional $\Psi_2$, which is
provided by the expression of the total potential energy, provides the level of
repulsion between the particles of this kind of point process. In the next
theorem, this level of repulsion appears explicitly in the the estimation of its
distance with an adapted Poisson point process.

\begin{theo}\label{Supertheo6}

  Let $\epsilon\in\bbbr_+$ and $\Phi$ be a Gibbs point process on $\XX$ with
  temperature parameter $\theta>0$, partition function $C(\theta)$ and total
  potential energy

$$U(x_1,\dots,x_k)=\displaystyle\sum_{i=1}^k \Psi_1(x_i)+\displaystyle\sum_{i=1}^{k-1}\displaystyle\sum_{j=i+1}^k \Psi_2(x_i,x_j),$$

such that $\displaystyle\int_\XX e^{-\theta\Psi_1(x)}\d x<+\infty$,
$\Psi_2\geq0$ and $\|\Psi_2\|_\infty\leq\epsilon$.

Then,

$$\mmmd_{TV}^*(\Phi,\zeta_M)\leq (M(\XX))^2 \theta\epsilon,$$

where $\zeta_M$ is the Poisson point process on $\XX$ with intensity measure
$M(\d x)=e^{-\theta\Psi_1(x)}\d x$.

\end{theo}

\subsection{Extension of a Kallenberg's theorem}\label{subsec_kallenberg}

The convergence of a sequence of thinned point processes given in Theorem
\ref{Supertheo3} of the previous subsection is established for rescaled point
processes, but this conclusion may actually be extended to a wider framework. In
the following theorem (Theorem 14.19 in \cite{kallenberg_foundations_1997}),
Kallenberg states the convergence in law of thinned point processes to a Cox
process.

\begin{theo}\label{Kallenbergtheo}

  Let $(\Phi_n)_{n\in\bbbn}$ be a sequence of point processes on $\XX$ and let
  $(p_n)_{n\in\bbbn}$ be a sequence of functions from $\XX$ to $[0,1)$ such that
  $(p_n)_{n\in\bbbn}$ tends to $0$ uniformly. Let $M$ be a random measure on
  $\XX$ and $\Gamma_M$ be a Cox point process directed by $M$. Then,
\begin{equation*}
  p_n\Phi_n\xrightarrow[n\to+\infty]{law}M \iff p_n\circ\Phi_n\xrightarrow[n\to+\infty]{law}\Gamma_M.
\end{equation*}

\end{theo}

Since the Polish distance provides a metric for the convergence in law, it
becomes conceivable to get a convergence speed for this last result. First of
all, we write in the following lemma the Polish distance between two Cox point
processes as an adapted Polish distance between its directing random measures.

\begin{lem}\label{lem1Supertheo4}

  Let $M_1,M_2$ be random measures on $\XX$ and $\Gamma_{M_1},\Gamma_{M_2}$ be
  Cox point processes directed by $M_1,M_2$ respectively. Then,
\begin{equation*}
  \mmmd_P({\Gamma_{M_1}},{\Gamma_{M_2}})=\overline{\mmmd}_P({M_1},{M_2}),
\end{equation*}

with $\overline{\mmmd}_P$ denoting the Polish distance on $\bbbm_1'$ associated
to $g=(g_k)_{k\in\bbbn}$ defined, for any $k\in\bbbn$ and any $\varphi\in\bbbm$,
by $g_k(\varphi)=\Esp[f_k(\zeta_\varphi)]$, where $\zeta_\varphi$ is a Poisson
point process on $\XX$ with intensity measure $\varphi$.

\end{lem}

In order to apply Theorem \ref{Steinprop}, we need now to determine a version of
the Papangelou intensity of a thinned configuration, which is given in the next
lemma.

\begin{lem}\label{lem2Supertheo4}

  Let  $\varphi\in N_\XX$ and a measurable function $p:\XX\to[0,1)$. Then, a
  version $c$ of the Papangelou intensity of $p\circ\varphi$ with respect to the
  measure $p(x)\varphi(\d x)$ is provided for any $x\in\XX$ and any $\eta\in
  N_\XX$ by
\begin{equation*}
  c(x,\eta)=\textbf{1}_{\{x\in\varphi\setminus\eta\}}\dfrac{1}{1-p(x)}.
\end{equation*}

\end{lem}

The previous elements lead to a key result: a thinned point process may be seen
as an approximation of a Cox process, which is formally stated in the following
lemma.

\begin{lem}\label{lem3Supertheo4}

  Let $\Phi$ be a point process on $\XX$ and $p$ be a function from $\XX$ to
  [0,1). Let $\Gamma_{p\Phi}$ be a Cox point process directed by $p\Phi$. Then,
\begin{equation*}
  \mmmd_{TV}^*(\P_{p\circ\Phi},\P_{\Gamma_{p\Phi}})\leq 2\Esp\Big[\displaystyle\sum_{x\in\Phi}p^2(x)\Big].
\end{equation*}

\end{lem}

We can then deduce in the next theorem an estimation of the investigated
distance.

\begin{theo}\label{Supertheo4}

  Let $\Phi$ be a point process on $\XX$ and let $p$ be a measurable function
  from $\XX$ to [0,1). Let $M$ be a random measure on $\XX$ and $\Gamma_M$ be a
  Cox point process directed by $M$. Then,
\begin{equation*}
  \mmmd_P({p\circ\Phi},{\Gamma_M})\leq 2\Esp\Big[\displaystyle\sum_{x\in\Phi}p^2(x)\Big]+\overline{\mmmd}_P({p\Phi},M),
\end{equation*}

with $\overline{\mmmd}_P$ denoting the Polish distance on $\bbbm_1'$ associated
to $g=(g_k)_{k\in\bbbn}$ defined, for any $n\in\bbbn$ and any $\varphi\in\bbbm$,
by $g_k(\varphi)=\Esp[f_k(\zeta_\varphi)]$, where $\zeta_\varphi$ is a Poisson
point process on $\XX$ with intensity measure $\varphi$.

\end{theo}


Under the assumptions of Theorem \ref{Supertheo4}, it is actually possible to
show that
\begin{equation*}
  \mmmd_{TV}^*({p\circ\Phi},{\Gamma_M})\leq 2\Esp\Big[\displaystyle\sum_{x\in\Phi}p^2(x)\Big]+\mmmd_{TV}^*({p\Phi},M).
\end{equation*}

However, the random measure $p\Phi$ has almost surely a discrete support, and
this implies that we cannot suppose that the quantity $\mmmd_{TV}^*({p\Phi},M)$
is close to $0$ in the general case, in particular when $M$ admits almost surely
a density with respect to the measure $\ell$. That is the reason why we choose
to use the Polish distance instead of a stronger distance for this last
convergence result.


\section{Proofs}\label{sec_proofs}


\subsection{Proofs of Section \ref{sec_steins_method}} \label{subsec_proof3}
\begin{proof}[Proof of Theorem \ref{bienunsemigroup}]
  For any measurable and bounded function $F:\widehat{N}_\XX\to\bbbr$ and any
  $\phi\in \widehat{N}_\XX$, since thinning is associative,
  \begin{equation*}
    P_s(P_tF)(\phi)= \int_{\widehat{N}_\XX} \int_{\widehat{N}_\XX} F(e^{-(t+s)}\circ \phi + e^{-s}\circ (1-e^{-t})\circ \psi + (1-e^{-
      s})\circ \eta ) \P_\zeta(\d\psi)\P_\zeta(\d \eta).
  \end{equation*}
  Furthermore, since
  \begin{equation*}
    e^{-s}(1-e^{-t}) + (1-e^{-s})=1-e^{-(t+s)},
  \end{equation*}

by the invariance property of the Poisson point process distribution, we deduce
that
\begin{equation*}
  e^{-s}\circ (1-e^{-t})\circ\zeta^{(1)} \\ + (1-e^{-s})\circ\zeta^{(2)} \stackrel{\Dcal}{=}(1-e^{-(t+s)})\circ\zeta,
\end{equation*}

where $\zeta^{(1)}$ and $\zeta^{(2)}$ are independent copies of $\zeta$.
Hence,
\begin{equation*}
  P_s(P_tF)(\phi)=\int_{\widehat{N}_\XX} F(e^{-(t+s)}\circ \phi+ (1-e^{-(t+s)})\circ \eta)\P_\zeta(\d \eta)
\end{equation*}

and the proof is thus complete.
\end{proof}
\begin{proof}[Proof of Theorem \ref{gradientprop0}]
  By the Mecke formula applied to $\zeta$, for any measurable function
  $u:\XX\times N_\XX\to\bbbr_+$,
\begin{eqnarray*}
  \int_{\XX}\Esp[D_x F(\zeta)u(x,\zeta)]M(\d x)&=&\Esp\Big[F(\zeta)\displaystyle\sum_{x\in\zeta}u(x,\zeta\setminus x)\Big]-\int_{\XX}\Esp[F(\zeta)u(x,\zeta)]M(\d x) \\
                                               &=&\Esp\Big[F(\zeta)\Big(\displaystyle\sum_{x\in\zeta}u(x,\zeta\setminus x)-\int_{\XX}u(x,\zeta)M(\d x)\Big)\Big].
\end{eqnarray*}
Hence, if $F(\phi)=0$ $\P_{\Phi}(\d\zeta)$-a.s., then $D_x F(\phi)=0$
$(M\otimes\P_{\zeta})(\d x,\d\phi)$-a.s., as expected.
\end{proof}
\begin{proof}[Proof of Theorem \ref{generator}]
  For any measurable and bounded function $F:\widehat{N}_\XX\to\bbbr$ and any
  $\phi\in \widehat{N}_\XX$,
\begin{eqnarray*}
  \left.\frac{\d P_tF(\phi)}{\d t}\right|_{t=0}&=&\displaystyle\lim_{t\to0}\dfrac{1}{t}(P_tF(\phi)-P_0F(\phi)) \\
                                               &=&\displaystyle\lim_{t\to0}\dfrac{1}{t}(\Esp[F(e^{-t}\circ\phi+(1-e^{-t})\circ\zeta)]-F(\phi)),
\end{eqnarray*}

and, for any $t>0$,
\begin{eqnarray*}
  \Esp[F(e^{-t}\circ\phi+(1-e^{-t})\circ\zeta)]&=&p_{00}(t)F(\phi)+\displaystyle\sum_{x\in\phi}p_{01}^{(x)}(t)F(\phi\setminus x)\\
                                               &&+p_{10}(t)\int_{\XX}F(\phi+x)\dfrac{M(\d x)}{M(\XX)}+R(t),
\end{eqnarray*}

where for any $x\in\phi$,
\begin{eqnarray*}
  p_{00}(t)&=&\P(e^{-t}\circ\phi=\phi,(1-e^{-t})\circ\zeta=\varnothing)\\
           &=&\P(e^{-t}\circ\phi=\phi)\P((1-e^{-t})\circ\zeta=\varnothing)\\
           &=&e^{-t|\phi|}e^{-(1-e^{-t})M(\XX)},
\end{eqnarray*}

where the computation of the second factor is deduced from the fact that
$(1-e^{-t})\circ\zeta$ is a Poisson point process with intensity measure
$(1-e^{-t})M$,
\begin{eqnarray*}
  p_{01}^{(x)}(t)&=&\P(\phi\setminus(e^{-t}\circ\phi)=x,(1-e^{-t})\circ\zeta=\varnothing)\\
                 &=&\P(\phi\setminus(e^{-t}\circ\phi)=x)\P((1-e^{- t})\circ\zeta=\varnothing) \\
                 &=&(1-e^{-t})e^{-t(|\phi|-1)}e^{-(1-e^{- t})M(\XX)},
\end{eqnarray*}
\begin{eqnarray*}
  p_{10}(t)&=&\P(e^{-t}\circ\phi=\phi,|(1-e^{- t})\circ\zeta|=1)\\
           &=&\P(e^{-t}\circ\phi=\phi)\P(|(1-e^{- t})\circ\zeta|=1) \\
           &=&e^{-t|\phi|}(1-e^{- t})M(\XX)e^{-(1-e^{-t})M(\XX)},
\end{eqnarray*}
\begin{equation*}
  R(t)=\Esp[F(e^{-t}\circ\phi+(1-e^{- t})\circ\zeta)\textbf{1}_{|\phi\setminus(e^{-t}\circ\phi)|+|(1-e^{- t})\circ\zeta|\geq2}].
\end{equation*}
Then,
\begin{multline*}
  \dfrac{1}{t}\big(\Esp[F(e^{-t}\circ\phi+(1-e^{- t})\circ\zeta)]-F(\phi)\big)=\\
  \shoveleft{=\dfrac{1}{t}\Big(\displaystyle\sum_{x\in\phi}p_{01}^{(x)}(t)(F(\phi\setminus x)-F(\phi))}\\
  \shoveright{+p_{10}(t)\int_{\XX}F(\phi+x)-F(\phi)\dfrac{M(\d x)}{M(\XX)}-p_\infty(t)F(\phi)+R(t)\Big)}\\
  \shoveleft{=\displaystyle\sum_{x\in\phi}\dfrac{p_{01}^{(x)}(t)}{t}(F(\phi\setminus x)-F(\phi))}\\
  \shoveright{+\dfrac{p_{10}(t)}{t}\int_{\XX}F(\phi+x)-F(\phi)\dfrac{M(\d x)}{M(\XX)}-\dfrac{p_\infty(t)}{t}F(\phi)+\dfrac{R(t)}{t},}\\
\end{multline*}

where
\begin{eqnarray*}
  p_\infty(t)&=&\P(|\phi\setminus(e^{-t}\circ\phi)|+|(1-e^{-t})\circ\zeta|\geq2)\\
             &=&1-\Big(p_{00}(t)+\displaystyle\sum_{x\in\phi}p_{01}^{(x)}(t)+p_{10}(t)\Big).
\end{eqnarray*}
Since for any
$x\in\phi$, $$\displaystyle\lim_{t\to0}\frac{p_{01}^{(x)}(t)}{t}=1,\
\displaystyle\lim_{t\to0}\frac{p_{10}(t)}{t}=M(\XX)\ \text{and}\
\displaystyle\lim_{t\to0}\frac{1-p_{00}(t)}{t}=|\phi|+M(\XX),$$

we get that $$\displaystyle\lim_{t\to0}\frac{p_\infty(t)}{t}=0,$$

then by boundedness of $F$ that $$\displaystyle\lim_{t\to0}\frac{R(t)}{t}=0,$$

hence the first result. The second result is a consequence of the Mecke formula
for a Poisson point process.
\end{proof}
\begin{proof}[Proof of Lemma \ref{semigroupgradientcom}]
  For any $t\in\bbbr_+$, any $x\in\XX$, any measurable and bounded function
  $F:\widehat{N}_\XX\to\bbbr$ and any $\phi\in \widehat{N}_\XX$, from the
  definitions of $D_x$ and $P_t$,
\begin{eqnarray*}
  D_x P_tF(\phi)&=&P_tF(\phi+x)-P_tF(\phi) \\
                &=&\Esp[F(e^{-t}\circ(\phi+x)+(1-e^{- t})\circ\zeta)-F(e^{-t}\circ\phi+(1-e^{-t})\circ\zeta)].
\end{eqnarray*}
Hence, since thinning is distributive with respect to sum,
\begin{equation*}
  D_x P_tF(\phi)=\Esp[F(e^{-t}\circ\phi+e^{-t}\circ x+(1-e^{-t})\circ\zeta)-F(e^{-t}\circ\phi+(1-e^{-t})\circ\zeta)],
\end{equation*}

and then, since $\P(e^{-t}\circ x=x)=1-\P(e^{-t}\circ x=\varnothing)=e^{-t}$, it
follows that
\begin{equation*}
  D_x P_tF(\phi)=e^{-t}P_tD_{x}F(\phi).
\end{equation*}
The proof is thus complete.
\end{proof}
\begin{proof}[Proof of Lemma \ref{ergodic}]
  For any $F\in\Lip_1(\widehat{N}_\XX,\mmmd_{TV})$, $t\in\bbbr_+$ and $\phi\in
  \widehat{N}_\XX$,
\begin{multline*}
  \big|P_tF(\phi)-\Esp[F(\zeta)]\big|\leq \big|P_tF(\phi)-P_tF(\varnothing)\big|+\big|P_tF(\varnothing)-\Esp[F(\zeta)]\big| \\
  \shoveleft{=\big|\Esp[F(e^{-t}\circ\phi+(1-e^{-t})\circ\zeta)]-\Esp[F((1-e^{-t})\circ\zeta)]\big|}\\
  \shoveright{+\big|\Esp[F((1-e^{-t})\circ\zeta)]-\Esp[F(\zeta)]\big|.}\\
\end{multline*}
On one hand, since $F\in\Lip_1(\widehat{N}_\XX,\mmmd_{TV})$,
\begin{eqnarray*}
  \big|\Esp[F(e^{-t}\circ\phi+(1-e^{-t})\circ\zeta)]-\Esp[F((1-e^{-t})\circ\zeta)]\big|&\leq&\Esp[\mmmd_{TV}(e^{-t}\circ\phi,\varnothing)] \\
                                                                                       &=&\Esp[|e^{-t}\circ\phi|],
\end{eqnarray*}

and, since $|e^{-t}\circ\phi|$ has a binomial distribution with parameters
$|\phi|$ and $e^{-t}$,
\begin{equation*}
  \big|\Esp[F(e^{-t}\circ\phi+(1-e^{-t})\circ\zeta)]-\Esp[F((1-e^{-t})\circ\zeta)]\big|\leq e^{-t}|\phi|.
\end{equation*}
On the other hand,
\begin{equation*}
  \Esp[F(\zeta)]=\Esp[F((1-e^{-t})\circ\zeta+e^{-t}\circ\zeta)],
\end{equation*}

then, since $F\in\Lip_1(\widehat{N}_\XX,\mmmd_{TV})$,
\begin{eqnarray*}
  \big|\Esp[F((1-e^{-t})\circ\zeta)]-\Esp[F(\zeta)]\big|&\leq&\Esp[\mmmd_{TV}(e^{-t}\circ\zeta,\varnothing)] \\
                                                        &=&\Esp[|e^{-t}\circ\zeta|],
\end{eqnarray*}

and, since $|e^{-t}\circ\zeta|$ has a Poisson distribution with parameter
$e^{-t}M(\XX)$,
\begin{equation*}
  \big|\Esp[F(e^{-t}\circ\phi+(1-e^{-t})\circ\zeta)]-\Esp[F((1-e^{-t})\circ\zeta)]\big|\leq e^{-t}M(\XX),
\end{equation*}

which concludes this proof.
\end{proof}
\begin{proof}[Proof of Theorem \ref{SDR}]
  For any $F\in\Lip_1(\widehat{N}_\XX,\mmmd_{TV})$ and any $\phi\in
  \widehat{N}_\XX$, from the definition of $L$,
\begin{equation*}
  \int_0^{+\infty} LP_sF(\phi)\d s=\int_0^{+\infty} \left.\Big(\dfrac{\d P_t(P_sF)}{\d t}\Big)\right|_{t=0}(\phi)\d s.
\end{equation*}
Hence, since $(P_t)_{t\geq0}$ is a semi-group,
\begin{equation*}
  \int_0^{+\infty} LP_sF(\phi)\d s=\int_0^{+\infty} \left.\Big(\dfrac{\d P_{t+s}F}{\d t}\Big)\right|_{t=0}(\phi)\d s
\end{equation*}

and it yields
\begin{eqnarray*}
  \int_0^{+\infty} LP_sF(\phi)\d s&=&\int_0^{+\infty} \dfrac{\d P_{s}F}{\d s}(\phi)\d s \\
                                  &=&\displaystyle\lim_{s\to+\infty} P_sF(\phi)-P_0F(\phi).
\end{eqnarray*}
Then, by Lemma \ref{ergodic},
\begin{equation*}
  \int_0^{+\infty} LP_sF(\phi)\d s=\Esp[F(\zeta)]-F(\phi).
\end{equation*}
The proof is thus complete.
\end{proof}
\begin{proof}[Proof of Theorem \ref{Steinprop}]
  For any $F\in\Lip_1(\widehat{N}_\XX,\mmmd_{TV})$, by Theorem \ref{SDR},
\begin{equation*}
  \Esp[F(\zeta)]-\Esp[F(\Phi)]=\Esp\Big[\int_0^{+\infty}LP_sF(\Phi)\d s\Big].
\end{equation*}
Then, from the expression of the generator $L$,
\begin{equation*}
  \Esp[F(\zeta)]-\Esp[F(\Phi)]=\int_0^{+\infty}\Esp\Big[\int_{\XX}D_x P_sF(\Phi)M(\d x)\Big]-\Esp\Big[\displaystyle\sum_{y\in\Phi}P_sF(\Phi)-P_sF(\Phi\setminus y)\Big]\d s
\end{equation*}

and then, by the definition of the Papangelou intensity,
\begin{eqnarray*}
  \Esp[F(\zeta)]-\Esp[F(\Phi)]&=&\int_0^{+\infty}\Esp\Big[\int_{\XX}D_x P_sF(\Phi)m(x)\d x)\Big]-\Esp\Big[\int_{\XX}c(x,\Phi)D_x P_sF(\Phi)\d x\Big]\d s\\
                              &=&\int_0^{+\infty}\Esp\Big[\int_{\XX}D_x P_sF(\Phi)(m(x)-c(x,\Phi))\d x\Big]\d s.
\end{eqnarray*}
Thus, by Lemma \ref{semigroupgradientcom},
\begin{equation*}
  \Esp[F(\zeta)]-\Esp[F(\Phi)]=\int_0^{+\infty}e^{-s}\Esp\Big[\int_{\XX}P_sD_{x}F(\Phi)(m(x)-c(x,\Phi))\d x\Big]\d s
\end{equation*}

and then, since $F\in\Lip_1(\widehat{N}_\XX,\mmmd_{TV})$ and $\|P_s\|\leq1$,
\begin{eqnarray*}
  \big|\Esp[F(\zeta)]-\Esp[F(\Phi)]\big|&\leq&\int_0^{+\infty}e^{-s}\Esp\Big[\int_{\XX}|D_{x}F(\Phi)||m(x)-c(x,\Phi)|\d x\Big]\d s\\
                                        &\leq&\int_{\XX}\Esp[|m(x)-c(x,\Phi)|]\d x.
\end{eqnarray*}
The proof is thus complete.
\end{proof}
\subsection{Proofs of Section \ref{sec_papangelou}} \label{subsec_proofs4}
\begin{proof}[Proof of Lemma \ref{repulsivelem1}]
  On one hand, by equation \ref{Papangelouprop2}, for any $x\in\XX$,
  $p_{0}\rho(x)=p_{0}\Esp[c(x,\Phi)]$, then, since $\Phi$ is repulsive,
  $p_{0}\rho(x)\leq p_{0}c(x,\varnothing)$. On the other hand, still by Theorem
  \ref{Papangelouprop2}, for any $x\in\XX$, $$\rho(x)=\Esp[c(x,\Phi)]\geq
  p_{0}c(x,\varnothing)$$

and it follows from both last inequalities that
\begin{equation*}
  |p_{0}c(x,\varnothing)-p_{0}\rho(x)|\leq (1-p_{0})p_{0}c(x,\varnothing),
\end{equation*}

hence, the result.
\end{proof}
\begin{proof}[Proof of Lemma \ref{repulsivelem2}]
  For any $x\in\XX$, by the triangle inequality,
\begin{equation*}
  \Esp[|c(x,\Phi)-\rho(x)|]\leq\Esp[|c(x,\Phi)-c(x,\varnothing)|]+\Esp[|c(x,\varnothing)-\rho(x)|].
\end{equation*}
Since $\Phi$ is weakly repulsive and since, in this case,
$\rho(x)=\Esp[c(x,\Phi)]\leq c(x,\varnothing)$, we deduce that
\begin{equation*}
  \Esp[|c(x,\Phi)-\rho(x)|]\leq\Esp[c(x,\varnothing)-c(x,\Phi)]+c(x,\varnothing)-\rho(x).
\end{equation*}
Hence, still since $\rho(x)=\Esp[c(x,\Phi)]$,
\begin{equation*}
  \Esp[|c(x,\Phi)-\rho(x)|]\leq2(c(x,\varnothing)-\rho(x)).
\end{equation*}
The proof is thus complete.
\end{proof}
\begin{proof}[Proof of Lemma \ref{Papangeloulem1}]
  This equation is deduced by applying the formula which defines the Papangelou
  intensity for $u:\XX\times\N_\XX\to\bbbr_+$ given for any $x\in\XX$ and
  $\phi\in\N_\XX$ by
\begin{equation*}
  u(x,\phi)=\textbf{1}_{\{\phi=\varnothing\}},
\end{equation*}

which concludes the proof.
\end{proof}
\begin{proof}[Proof of Theorem \ref{papreduction}]
  For any measurable function $u:\widehat{N}_\XX\to\bbbr_+$, by definition of
  $\Phi_{|\Lambda}$,
\begin{equation*}
  \Esp\Big[\displaystyle\sum_{x\in\Phi_{|\Lambda}}u(x,\Phi_{|\Lambda}\setminus x)\Big]=\Esp\Big[\displaystyle\sum_{x\in\Phi}u(x,(\Phi\setminus x)\cap\Lambda)\textbf{1}_{x\in\Lambda}\Big].
\end{equation*}
Then, by the definition of the Papangelou intensity,
\begin{equation*}
  \Esp\Big[\displaystyle\sum_{x\in\Phi_{|\Lambda}}u(x,\Phi_{|\Lambda}\setminus x)\Big]=\int_\XX\Esp\Big[c(x,\Phi)u(x,\Phi\cap\Lambda)\textbf{1}_{x\in\Lambda}\Big]\d x,
\end{equation*}

and the expected result is derived.
\end{proof}
\begin{proof}[Proof of Theorem \ref{papsuperposition}]
  Denoting $k_{[n]}=k_1+\dots+k_n$, for any measurable function
  $u:\XX\times\N_\XX\to\bbbr_+$,
\begin{equation*}
  \Esp\Big[\displaystyle\sum_{y\in\Phi_1+\dots+\Phi_n} u\Big(y,\displaystyle\sum_{i=1}^n\Phi_i \setminus \{y\}\Big)\Big]=\displaystyle\sum_{i=1}^n \Esp\Big[\displaystyle\sum_{y\in\Phi_i}u\Big(y,\displaystyle\sum_{i=1}^n\Phi_i \setminus \{y\}\Big)\Big].
\end{equation*}
Then, applying the definition of the Papangelou intensity for each $\Phi_i$,
\begin{equation*}
  \Esp\Big[\displaystyle\sum_{y\in\Phi_1+\dots+\Phi_n} u\Big(y,\displaystyle\sum_{i=1}^n\Phi_i \setminus \{y\}\Big)\Big]= \displaystyle\sum_{i=1}^n \Esp\Big[\int_E u(y,\displaystyle\sum_{i=1}^n\Phi_i)c_i(y,\Phi_i)\d y\Big],
\end{equation*}

from which we deduce that
\begin{equation*}
  \Esp\Big[\displaystyle\sum_{y\in\Phi_1+\dots+\Phi_n} u\Big(y,\displaystyle\sum_{i=1}^n\Phi_i \setminus \{y\}\Big)\Big]= \Esp\Big[\int_E u\Big(y,\displaystyle\sum_{i=1}^n\Phi_i\Big) \displaystyle\sum_{i=1}^n c_i(y,\Phi_i)\d y\Big],
\end{equation*}

which yields the identity verified by the Papangelou intensities.
\end{proof}
\begin{proof}[Proof of Corollary \ref{corpapsuperposition}]
  For any $i\in\{1,\dots,n\}$, let $c_i$ be a version of the Papangelou
  intensity of $\Phi_i$ such that $\Phi_i$ is weakly repulsive according to
  $c_i$. Then, by Theorem \ref{papsuperposition}, one can find a version $c$ of
  the Papangelou intensity of the superposition verifying, for any $x\in\XX$:
\begin{equation*}
  c(x,\varnothing)=\displaystyle\sum_{i=1}^n c_i(x,\varnothing)\geq\displaystyle\sum_{i=1}^n c_i(x,\Phi_i)=c(x,\displaystyle\sum_{i=1}^n \Phi_i)\ \text{a.s.},
\end{equation*}

from which we conclude the proof.
\end{proof}
\begin{proof}[Proof of Theorem \ref{papthinning1}]
  For any measurable function $u:\XX\times\N_\XX\to\bbbr_+$, one has:
  \begin{eqnarray*}
    \Esp\Big[\displaystyle\sum_{x\in\beta\circ\Phi}u(x,\beta\circ\Phi\setminus x)\Big]&=&\Esp\Big[\displaystyle\sum_{x\in\Phi}u(x,\beta\circ\Phi\setminus x)\textbf{1}_{x\in\beta\circ\Phi}\Big]\\
                                                                                      &=&\Esp\Big[\displaystyle\sum_{x\in\Phi}\displaystyle\sum_{\tau\subset\Phi}u(x,\tau\setminus x)\textbf{1}_{x\in\tau}\textbf{1}_{\tau=\beta\circ\Phi}\Big],
  \end{eqnarray*}

then, conditioning with respect to $\Phi$,
\begin{eqnarray*}
  \Esp\Big[\displaystyle\sum_{x\in\beta\circ\Phi}u(x,\beta\circ\Phi\setminus x)\Big]&=&\Esp\Big[\Esp\Big[\displaystyle\sum_{x\in\Phi}\displaystyle\sum_{\tau\subset\Phi}u(x,\tau\setminus x)\textbf{1}_{x\in\tau}\textbf{1}_{\tau=\beta\circ\Phi}\ |\ \Phi\Big]\Big] \\
                                                                                    &=&\Esp\Big[\displaystyle\sum_{x\in\Phi}\displaystyle\sum_{\tau\subset\Phi}\P(\tau=\beta\circ\Phi\ |\ \Phi)u(x,\tau\setminus x)\textbf{1}_{x\in\tau}\Big].
\end{eqnarray*}
Since, for any $\tau\subset\phi$,
$\P(\tau=\beta\circ\phi)=\Big(\prod_{x\in\tau}\beta(x)\Big)\Big(\prod_{x\in\phi\setminus\tau}\big(1-\beta(x)\big)\Big)$,
one gets:
\begin{eqnarray*}
  \Esp\Big[\displaystyle\sum_{x\in\beta\circ\Phi}u(x,\beta\circ\Phi\setminus x)\Big]&=&\Esp\Big[\displaystyle\sum_{x\in\Phi}\displaystyle\sum_{\tau\subset\Phi}\Big(\prod_{y\in\tau}\beta(y)\Big)\Big(\prod_{y\in\Phi\setminus\tau}\big(1-\beta(y)\big)\Big)u(x,\tau\setminus x)\textbf{1}_{x\in\tau}\Big] \\
                                                                                    &=&\Esp\Big[\displaystyle\sum_{x\in\Phi}\displaystyle\sum_{\tau\subset\Phi\setminus x}\beta(x)\Big(\prod_{y\in\tau}\beta(y)\Big)\Big(\prod_{y\in(\Phi\setminus x)\setminus\tau}\big(1-\beta(y)\big)\Big)u(x,\tau)\Big].
\end{eqnarray*}
Then, from the definition of the Papangelou intensity,
\begin{eqnarray*}
  &&\Esp\Big[\displaystyle\sum_{x\in\beta\circ\Phi}u(x,\beta\circ\Phi\setminus x)\Big]=\\
  &=&\int_\XX\Esp\Big[c(x,\Phi)\displaystyle\sum_{\tau\subset\Phi}\beta(x)\Big(\prod_{y\in\tau}\beta(y)\Big)\Big(\prod_{y\in\Phi\setminus\tau}\big(1-\beta(y)\big)\Big)u(x,\tau)\Big]\d x.
\end{eqnarray*}
The previous arguments yield
\begin{eqnarray*}
  \Esp\Big[\displaystyle\sum_{x\in\beta\circ\Phi}u(x,\beta\circ\Phi\setminus x)\Big]&=&\int_\XX\Esp\Big[\beta(x)c(x,\Phi)\displaystyle\sum_{\tau\subset\Phi}\P(\beta\circ\Phi=\tau\ |\ \Phi)u(x,\tau)\Big]\d x\\
                                                                                    &=&\int_\XX\Esp\Big[\beta(x)c(x,\Phi)\displaystyle\sum_{\tau\subset\Phi}\textbf{1}_{\beta\circ\Phi=\tau}u(x,\tau)\Big]\d x \\
                                                                                    &=&\int_\XX\Esp[\beta(x)c(x,\Phi)u(x,\beta\circ\Phi)]\d x,
\end{eqnarray*}

hence, the result.
\end{proof}
\begin{proof}[Proof of Theorem \ref{paprescaling}]
  By formula \ref{Papangelouprop1},
\begin{equation*}
  c^{(\epsilon)}(x,\phi)=\dfrac{j^{(\epsilon)}(x\phi)}{j^{(\epsilon)}(\phi)}\textbf{1}_{j^{(\epsilon)}(\phi)\neq0},
\end{equation*}

where $j^{(\epsilon)}$ is the Janossy function of $\Phi^{(\epsilon)}$, and then
the expected result is deduced, still by formula (\ref{Papangelouprop1}).
\end{proof}
\begin{proof}[Proof of Theorem \ref{purelyrandomprop2bis}]
  The Janossy function $j$ of a purely random point process is given for any
  $n\in\bbbn_0$ and any $x_1,\dots,x_n\in\XX$ by
  \begin{equation*}
    j(x_1,\dots,x_n)=p_n n!q(x_1)\dots q(x_n)
  \end{equation*}
  and then we deduce the expression of the Papangelou intensity from formula
  (\ref{Papangelouprop1}), which provides the link between Janossy function and
  Papangelou intensity. In particular, this implies that $\Phi$ is repulsive if
  and only if, for any $n\in\bbbn_0$ and any $x\in\XX$,
\begin{equation*}
  (n+2)\dfrac{p_{n+2}}{p_{n+1}}q(x)\leq(n+1)\dfrac{p_{n+1}}{p_n}q(x),
\end{equation*}

which is equivalent to the expected assertion, and that $\Phi$ is weakly
repulsive if and only if, for any $n\in\bbbn_0$ and any $x\in\XX$,
\begin{equation*}
  (n+1)\dfrac{p_{n+1}}{p_{n}}q(x)\leq\dfrac{p_{1}}{p_0}q(x),
\end{equation*}

hence, the result.
\end{proof}
\begin{proof}[Proof of Theorem \ref{CondPPPprop2bis}]
  A version of the Papangelou intensity is deduced from the Janossy function by
  formula (\ref{Papangelouprop1}) and the Janossy function $j$ of a finite
  conditional Poisson point process is given for any $\phi\in\widehat{\N}_\XX$
  by
\begin{equation}\label{CondPPPprop2}
  j(\phi)=\dfrac{e^{-M(\XX)}}{p_C}\displaystyle\prod_{x\in\phi}m(x)\textbf{1}_C(\phi),
\end{equation}

where $p_C=\bbbp(\Phi\in C)$ and $\Phi$ is the Poisson point process associated
to $\Phi_C$, which provides the expected expression. As a consequence, $\Phi$ is
repulsive if and only if, for any $x,x_1,\dots,x_n,x_{n+1}\in\XX$,
\begin{equation*}
  m(x)\textbf{1}_{\{x_1,\dots,x_{n+1},x\}\in C}\textbf{1}_{\{x_1,\dots,x_{n+1}\}\in C}\leq m(x)\textbf{1}_{\{x_1,\dots,x_n,x\}\in C}\textbf{1}_{\{x_1,\dots,x_n\}\in C}.
\end{equation*}
Hence, if $C$ is decreasing, then this last hypothesis is verified, and $\Phi$
is repulsive.
\end{proof}
\begin{proof}[Proof of Theorem \ref{Gibbsprop1}]
  The expression of the Papangelou intensity is deduced from the definition of a
  Gibbs point process and from formula (\ref{Papangelouprop1}). In order to show
  that $\Phi$ is repulsive, one can observe that, for any
  $x,x_1,\dots,x_n,x_{n+1}\in\XX$,
\begin{multline*}
  \big(U(x_1,\dots,x_n,x_{n+1},x)-U(x_1,\dots,x_n,x_{n+1})\big)-\big(U(x_1,\dots,x_n,x)-U(x_1,\dots,x_n)\big)= \\
  \shoveleft{=\big(\Psi_1(x)+\displaystyle\sum_{r=1}^{n+2}\displaystyle\sum_{1\leq i_1<\dots< i_{r-1}\leq n+1}\Psi_r(x_{i_1},\dots,x_{i_{r-1}},x)\big)}\\
  \shoveright{-\big(\Psi_1(x)+\displaystyle\sum_{r=1}^{n+1}\displaystyle\sum_{1\leq i_1<\dots< i_{r-1}\leq n}\Psi_r(x_{i_1},\dots,x_{i_{r-1}},x)\big)} \\
  \shoveleft{=\Psi_{n+2}(x_1,\dots,x_{n+1},x)+\displaystyle\sum_{r=2}^{n+1}\displaystyle\sum_{1\leq i_1<\dots< i_{r-2}\leq n}\Psi_r(x_{i_1},\dots,x_{i_{r-2}},x_{n+1},x)} \\
  \shoveleft\geq0.\\
\end{multline*}
The proof is thus complete.
\end{proof}
\subsection{Proofs of Section \ref{sec_applications}} \label{subsec_proofs5}
\begin{proof}[Proof of Theorem \ref{Superpoisson}]
  For any $i\in\{1,2\}$, the Papangelou intensity of $\zeta_i$ with respect to
  $M_1+M_2$ is given by $\frac{\d M_i}{\d(M_1+M_2)}$. The result is deduced by
  combining Theorem \ref{Steinprop} and equation (\ref{remdTV}).
\end{proof}
\begin{proof}[Proof of Theorem \ref{Supercox}]
  Using the notations of the definition of Kantorovich-Rubinstein,
\begin{eqnarray*}
  \mmmd_{TV}^*(\Gamma_1,\Gamma_2)&=& \inf_{\bC\in\Sigma(\P_{\Gamma_1},\P_{\Gamma_2})}\int_{\N_\XX\times\N_\XX}\mmmd_{TV}(\omega_1,\omega_2)\bC(\d(\omega_1,\omega_2))\\
                                 &\leq& \inf_{\bC\in\Sigma(\P_{M_1},\P_{M_2})} \int_{\bbbm\times\bbbm} \mmmd_{TV}^*(\zeta_{\varphi_1},\zeta_{\varphi_2}) \bC(\d(\varphi_1,\varphi_2)).
\end{eqnarray*}
By Theorem \ref{Superpoisson}, it follows as expected that
\begin{equation*}
  \mmmd_{TV}^*(\Gamma_1,\Gamma_2)\leq \inf_{\bC\in\Sigma(\P_{M_1},\P_{M_2})} \int_{\bbbm\times\bbbm} \mmmd_{TV}(\varphi_1,\varphi_2)\bC(\d(\varphi_1,\varphi_2)),
\end{equation*}

from which we conclude the proof.
\end{proof}
\begin{proof}[Proof of Theorem \ref{Supertheo5}]
  The point process $\Phi$ has a Papangelou intensity $c$ given for any
  $x,x_1,\dots,x_n\in\XX$ by:
$$c(x,\{x_1,\dots,x_n\})=\dfrac{n+1}{M(\XX)}\dfrac{p_{n+1}}{p_n}m(x).$$
Then, by Theorem \ref{Steinprop},
\begin{equation*}
  \mmmd_{TV}^*(\Phi,\zeta_M)\leq\int_\XX \displaystyle\sum_{n=0}^{+\infty} p_n\Big|\dfrac{n+1}{M(\XX)}\frac{p_{n+1}}{p_n}m(x)-m(x)\Big|\d x,
\end{equation*}

and then
\begin{equation*}
  \mmmd_{TV}^*(\Phi,\zeta_M)\leq\displaystyle\sum_{n=0}^{+\infty}\big|(n+1)p_{n+1}-M(\XX)p_n\big|.
\end{equation*}
The proof is thus complete.
\end{proof}
\begin{proof}[Proof of Theorem \ref{Supertheo7}]
  By Theorem \ref{Steinprop} and from the expression of the Papangelou intensity
  of $\Phi_C$ (Theorem \ref{CondPPPprop2bis}),
\begin{eqnarray*}
  \mmmd_{TV}^*(\Phi_C,\Phi)&\leq&\int_\XX\Esp[|m(x)-m(x)\textbf{1}_{C}(\Phi_C x)\textbf{1}_{C}(\Phi_C)|]\d x \\
                           &=&\int_\XX m(x)\bbbp(\Phi_C x\notin C)\d x,
\end{eqnarray*}

since $\Phi_C\in C$ almost surely.
\end{proof}
\begin{proof}[Proof of Corollary \ref{cor1Supertheo7}]
  By Theorem \ref{Supertheo7},
\begin{equation*}
  \mmmd_{TV}^*(\Phi_{R},\Phi)\leq\lambda\int_\XX\bbbp(\Phi_{R}x\notin C_R)\d x,
\end{equation*}

then, by formula (\ref{CondPPPprop2}),
\begin{multline*}
  \mmmd_{TV}^*(\Phi_{R},\Phi)\leq e^{-\lambda|\Lambda|}\dfrac{\lambda}{p_R}\\
  \times \sum_{k=0}^{+\infty}\dfrac{\lambda^k}{k!}\int_{\Lambda^{k+1}}\textbf{1}_{C_R^c}(\{x_1,\dots,x_k,x\})\textbf{1}_{C_R}(\{x_1,\dots,x_k\}) \d x_1\dots\d x_k\d x,
\end{multline*} 

and then, since $\textbf{1}_{C_R}\leq 1$ and
$\textbf{1}_{C_R^c}=1-\textbf{1}_{C_R}$,
\begin{multline*}
  \mmmd_{TV}^*(\Phi_{R},\Phi)\leq e^{-\lambda|\Lambda|}\dfrac{\lambda}{p_R}\\
  \times\displaystyle\sum_{k=0}^{+\infty}\dfrac{\lambda^k}{k!}\int_{\Lambda^{k+1}}(1-\textbf{1}_{C_R}(\{x_1,\dots,x_k,x\}))\d x_1\dots\d x_k\d x.
\end{multline*} 
Moreover, since
$\textbf{1}_{C_R}(\{x_1,\dots,x_k,x\})\geq\displaystyle\prod_{i=1}^k\textbf{1}_{\mmmd_\XX(x_i,x)\geq
  R}$, one has
\begin{multline*}
  \mmmd_{TV}^*(\Phi_{R},\Phi)\leq e^{-\lambda|\Lambda|}\dfrac{\lambda}{p_R}\displaystyle\sum_{k=0}^{+\infty}\dfrac{\lambda^k}{k!}\int_{\Lambda^{k+1}}(1-\displaystyle\prod_{i=1}^k\textbf{1}_{\mmmd_\XX(x_i,x)\geq R})\d x_1\dots\d x_k\d x\\
                             =e^{-\lambda|\Lambda|}\dfrac{\lambda}{p_R}\displaystyle\sum_{k=0}^{+\infty}\dfrac{\lambda^k}{k!}\int_{\Lambda^{k}}|\Lambda|-\Big(\int_\Lambda\displaystyle\prod_{i=1}^k\textbf{1}_{\mmmd_\XX(x_i,x)\geq R}\d x\Big)\d x_1\dots\d x_k,
\end{multline*} 

and then, since $V_d(R)$ is the volume of a ball of $\bbbr^d$ with radius $R$,
\begin{eqnarray*}
  \mmmd_{TV}^*(\Phi_{R},\Phi)&\leq& e^{-\lambda|\Lambda|}\dfrac{\lambda}{p_R}\displaystyle\sum_{k=0}^{+\infty}\dfrac{\lambda^k}{k!}\int_{\Lambda^{k}}|\Lambda|-(|\Lambda|-kV_d(R))\d x_1\dots\d x_k \\
                             &=&\dfrac{\lambda^2|\Lambda|}{p_{R}}V_d(R).
\end{eqnarray*}
The proof is thus complete.
\end{proof}
\begin{proof}[Proof of Corollary \ref{cor2Supertheo7}]
  By Theorem \ref{Supertheo7},
\begin{equation*}
  \mmmd_{TV}^*(\Phi_{N},\Phi)\leq\int_\XX\bbbp(\Phi_N x\notin C_N)m(x)\d x,
\end{equation*}

then, since by formula (\ref{CondPPPprop2}), for any $x\in\XX$,
\begin{multline*}
  \bbbp(\Phi_N x\notin C_N)\\=\dfrac{e^{-M(\XX)}}{p_N}\displaystyle\sum_{k=0}^{+\infty}\dfrac{1}{k!}\int_{\XX^{k+1}}\textbf{1}_{C_N^c}(\{x_1,\dots,x_k,x\})\textbf{1}_{C_N}(\{x_1,\dots,x_k\})\otimes_{j=1}^{k} M(\d x_j),
\end{multline*}

it yields
\begin{eqnarray*}
  \mmmd_{TV}^*(\Phi_{N},\Phi)&\leq&\dfrac{e^{-M(\XX)}}{p_N}\dfrac{1}{N!}\int_{\XX^{N+1}}m(x_1)\dots m(x_N)m(x)\d x_1\dots\d x_N\d x\\
                             &=&\dfrac{e^{-M(\XX)}}{p_N}\dfrac{1}{N!}(M(\XX))^{N+1}.
\end{eqnarray*}
The proof is thus complete.
\end{proof}
\begin{proof}[Proof of Theorem \ref{supertheo}]
  For any $k\in\bbbn_0$, we use the notation $p_{n,i,k}:=\P(|\Phi_{n,i}|=k)$. By
  Theorem \ref{Steinprop},
 
\begin{equation*}
  \mmmd_{TV}^*(\Phi_n,\zeta_M)\leq\int_\XX\Esp[|c_n(x,\Phi_n)-m(x)|]\d x.
\end{equation*}
Then, by Theorem \ref{papsuperposition}, $\mmmd_{TV}^*(\Phi_n,\zeta_M)\leq
R_n+\displaystyle\sum_{i=1}^n A_{n,i}$, where
 
\begin{eqnarray*}
  A_{n,i}&=&\int_\XX\Esp[|c_{n,i}(x,\Phi_{n,i})-\rho_{n,i}(x)|]\ell(\d x) \\
         &=&\displaystyle\sum_{k\geq0} \int_\XX\Esp[|c_{n,i}(x,\Phi_{n,i})-\rho_{n,i}(x)|{\bf 1}_{\{|\Phi_{n,i}|=k\}}]\ell(\d x) \\
         &=&B_{n,i}+C_{n,i}
\end{eqnarray*}

with
\begin{equation*}
  B_{n,i}=p_{n,i,0}\int_\XX|c_{n,i}(x,\varnothing)-\rho_{n,i}(x)|\ell(\d x),
\end{equation*}
\begin{equation*}
  C_{n,i}=\displaystyle\sum_{k\geq1} \int_\XX\Esp[|c_{n,i}(x,\Phi_{n,i})-\rho_{n,i}(x)|{\bf 1}_{\{|\Phi_{n,i}|=k\}}]\ell(\d x).
\end{equation*}
By Lemma \ref{Papangeloulem1},
\begin{equation*}
  p_{n,i,0}\int_\XX c_{n,i}(x,\varnothing) \ell(\d x)=p_{n,i,1}\leq(1-p_{n,i,0})
\end{equation*}

and by Lemma \ref{repulsivelem1} we get
\begin{equation*}
  B_{n,i}\leq (1-p_{n,i,0})^2.
\end{equation*}
Since $c_{n,i}(x,\Phi_{n,i})\leq c_{n,i}(x,\varnothing)$ and $\rho_{n,i}(x)\leq
c_{n,i}(x,\varnothing)$, we also have
\begin{equation*}
  C_{n,i}\leq\displaystyle\sum_{k\geq1} p_{n,i,k}\int_\XX c_{n,i}(x,\varnothing) \ell(\d x)=(1-p_{n,i,0})\int_\XX c_{n,i}(x,\varnothing) \ell(\d x)\leq(1-p_{n,i,0})^2,
\end{equation*}

and then we get
\begin{equation*}
  A_{n,i}\leq2(1-p_{n,i,0})^2\leq2\Big(\int_{\XX}\rho_{n,i}(x)\ell(\d x)\Big)^2
\end{equation*}

where the second equation is derived from the Markov inequality. Hence,
\begin{equation*}
  \mmmd_{TV}^*(\Phi_n,\zeta_M)\leq R_n+2n\Big(\displaystyle\max_{i\in\{1,\dots,n\}}\int_{\XX}\rho_{n,i}(x)\ell(\d x)\Big)^2,
\end{equation*}

from which we conclude the proof.
\end{proof}
\begin{proof}[Proof of Corollary \ref{supertheocor2}]
  Since the $(-1/n)$-determinantal point process $\Phi_n$ with kernel $K$ is the
  independent superposition of $n$ determinantal point processes with kernel
  $\frac{1}{n}K$, it follows By Remark \ref{supertheocor1} that, for any
  $n\in\bbbn$,
\begin{equation*}
  \mmmd_{TV}^*(\Phi_n,\zeta_M)\leq R_n+\dfrac{2C^2}{n},
\end{equation*}

where
\begin{equation*}
  R_n=\int_\XX \Big|\displaystyle\sum_{i=1}^n \dfrac{1}{n}K(x,x)-K(x,x)\Big|\ell(\d x)=0
\end{equation*}

and
\begin{equation*}
  C=\int_\XX K(x,x)\d x,
\end{equation*}

from which we can conclude.
\end{proof}
\begin{proof}[Proof of Corollary \ref{supertheocor3}]
  The result is obtained by applying Theorem \ref{supertheo} to
  $(\Phi_{n,i})_{1\leq i\leq n}$ such that for each $n\in\bbbn$ and
  $i\in\{1,\dots,n\}$, $\Phi_{n,i}=\{X_{n,i}\}\cap\Lambda$.
\end{proof}
\begin{proof}[Proof of Theorem \ref{supertheo1bis}]
  By Theorem \ref{Steinprop}, one has for any $n\in\bbbn$,
\begin{equation*}
  \mmmd_{TV}^*(\Phi_n,\zeta)\leq\int_\Lambda\Esp\big[\big|c_n(x,\Phi_n)-m(x)\big|\big]\d x,
\end{equation*}

where $c_n$ is the Papangelou intensity of $\Phi_n$. Combining Theorem
\ref{papsuperposition} for the Papangelou intensity of an independent
superposition and Theorem \ref{papthinning1} for the Papangelou intensity of a
thinning, it follows that
\begin{equation*}
  \mmmd_{TV}^*(\Phi_n,\zeta)\leq\int_\Lambda\Esp\Big[\Big|\displaystyle\sum_{k=1}^n\dfrac{1}{n}\Esp\big[c(x,\Phi^{(k)})\ \big|\ \dfrac{1}{n}\circ\Phi^{(k)}\big]-m(x)\Big|\Big]\d x.
\end{equation*}
Hence, by Jensen's inequality,
\begin{equation*}
  \mmmd_{TV}^*(\Phi_n,\zeta)\leq\int_\Lambda\sqrt{\Var\Big[\dfrac{1}{n}\displaystyle\sum_{k=1}^n\Esp\big[c(x,\Phi^{(k)})\ \big|\ \dfrac{1}{n}\circ\Phi^{(k)}\big]\Big]}\d x,
\end{equation*}

and, by some variance properties,
\begin{eqnarray*}
  \mmmd_{TV}^*(\Phi_n,\zeta)&\leq&\dfrac{1}{\sqrt{n}}\int_\Lambda\sqrt{\Var[\Esp[c(x,\Phi)\ |\ \dfrac{1}{n}\circ\Phi]]}\d x\\
                            &\leq&\dfrac{1}{\sqrt{n}}\int_\Lambda\sqrt{\Var[c(x,\Phi)]}\d x.
\end{eqnarray*}
By hypothesis, for any $x\in\Lambda$, $\Var[c(x,\Phi)]\leq K(x)$ and one deduces
the expected result.
\end{proof}
\begin{proof}[Proof of Theorem \ref{Supertheo3}]
  The family of determinantal point processes is stable with respect to several
  transformations: the reduction to a compact set, the thinning and the
  rescaling. Their corresponding kernels are respectively provided by formulas
  (\ref{reductiondeterminantal}), (\ref{thinningdeterminantal}) and
  (\ref{rescalingdeterminantal}). Combining these expressions, it follows that
  $\Phi_{\Lambda,\beta}$ is the determinantal point process with kernel
  $K_{\Lambda,\beta}$ defined by
  \begin{equation*} \label{beta} K_{\Lambda,\beta}:(x,y)\in \XX\times\XX\mapsto
    K\Big(\dfrac{x}{{\beta}^{\frac{1}{d}}},\dfrac{y}{{\beta}^{\frac{1}{d}}}\Big)\textbf{1}_{\Lambda\times\Lambda}(x,y).
  \end{equation*}
  By Theorem \ref{DPPprop1}, there exists a complete orthonormal basis $(h_j,\,
  {j\in\bbbn})$ of $ L^2(\XX,\ell;\bbbc)$ and a sequence $(\lambda_j,\,
  {j\in\bbbn})\subset [0,1]^\bbbn $ such that for any $x,y\in \XX$,
\begin{equation*} 
  K(x,y)=\displaystyle\sum_{j=1}^{+\infty}\lambda_j h_j(x)h_j(y). 
\end{equation*}
Then, for any $x,y\in \XX$,
\begin{eqnarray*}
  K_{\Lambda,\beta}(x,y)&=&K\Big(\dfrac{x}{{\beta}^{\frac{1}{d}}},\dfrac{y}{{\beta}^{\frac{1}{d}}}\Big)\textbf{1}_{\Lambda\times\Lambda}(x,y) \\
                        &=&\displaystyle\sum_{j=1}^{+\infty}\lambda_jh_j\Big(\dfrac{x}{{\beta}^{\frac{1}{d}}}\Big)\textbf{1}_{\Lambda}(x)h_j\Big(\dfrac{y}{{\beta}^{\frac{1}{d}}}\Big)\textbf{1}_{\Lambda}(y) \\
                        &=&\displaystyle\sum_{j=1}^{+\infty}\lambda_{\Lambda,\beta,j} {h}_{\Lambda,\beta,j}(x){h}_{\Lambda,\beta,j}(y),
\end{eqnarray*}

where, for any $j\in\bbbn$ and any $x\in\XX$,
$$Z_{\Lambda,\beta,j}^2=\int_{\beta^{-\frac{1}{d}}\Lambda}|h_j(y)|^2\d y,$$
$${h}_{\Lambda,\beta,j}(x)=\dfrac{1}{\sqrt{\beta}}Z_{\Lambda,\beta,j}^{-1}h_j\Big(\dfrac{x}{{\beta}^{\frac{1}{d}}}\Big)\textbf{1}_{\Lambda}(x),$$
$$\lambda_{\Lambda,\beta,j}=\lambda_j\beta Z_{\Lambda,\beta,j}^2.$$
By Theorem \ref{DPPprop1}, since, for any $j\in\bbbn$,
$\lambda_{\Lambda,\beta,j}<1$, one can associate to $K_{\Lambda,\beta}$ the
kernel $J_{\Lambda,\beta}$ such that for any $x,y\in\bbbr^d$,
\begin{equation*}
  J_{\Lambda,\beta}(x,y)=\displaystyle\sum_{j=1}^{+\infty}\dfrac{\lambda_{\Lambda,\beta,j}}{1-\lambda_{\Lambda,\beta,j}} {h}_{\Lambda,\beta,j}(x){h}_{\Lambda,\beta,j}(y),
\end{equation*}

and, by Theorem \ref{DPPpap}, for any $x\in\bbbr^d$,
\begin{equation*}
  J_{\Lambda,\beta}(x,x)=c_{\Lambda,\beta}(x,\varnothing).
\end{equation*}
In particular, still by Theorem \ref{DPPpap}, $\Phi_{\Lambda,\beta}$ is a
weakly repulsive point process, then, by Lemma \ref{repulsivelem2}, for any
$x\in\Lambda$ and $\phi\in\N_{\Lambda}$,
\begin{equation*}
  \Esp[|c_{\Lambda,\beta}(x,\phi)-\lambda|]\leq 2(c_{\Lambda,\beta}(x,\varnothing)-\lambda).
\end{equation*}

Then, by Theorem \ref{Steinprop},
\begin{equation*}
  \mmmd_{TV}^*(\Phi_{\Lambda,\beta},\zeta_{\Lambda,\lambda})\leq2\int_\Lambda \big(c_{\Lambda,\beta}(x,\varnothing)-\lambda\big)\, \d x.
\end{equation*}
By previous identities, one has
\begin{equation*}
  \int_\Lambda c_{\Lambda,\beta}(x,\varnothing)\d x=\displaystyle\sum_{j=1}^{+\infty}\dfrac{\lambda_{\Lambda,\beta,j}}{1-\lambda_{\Lambda,\beta,j}}.
\end{equation*}
Then, noting that
\begin{multline*}
  \int_\Lambda \lambda \d x=\int_\Lambda K\Big(\dfrac{x}{\beta^{\frac{1}{d}}},\dfrac{x}{\beta^{\frac{1}{d}}}\Big) \d x=\int_\Lambda K_{\Lambda,\beta}(x,x) \d x\\ =\int_\Lambda \displaystyle\sum_{j=1}^{+\infty}\lambda_{\Lambda,\beta,j}h_{\Lambda,\beta,j}^2(x) \d x=\displaystyle\sum_{j=1}^{+\infty}\lambda_{\Lambda,\beta,j},
\end{multline*}

one obtains
\begin{equation*}
  \mmmd_{TV}^*(\Phi_{\Lambda,\beta},\zeta_{\Lambda,\lambda})\leq2 \displaystyle\sum_{j=1}^{+\infty}\dfrac{\lambda_{\Lambda,\beta,j}}{1-\lambda_{\Lambda,\beta,j}}-\lambda_{\Lambda,\beta,j}=2 \displaystyle\sum_{j=1}^{+\infty}\dfrac{\lambda_{\Lambda,\beta,j}^2}{1-\lambda_{\Lambda,\beta,j}},
\end{equation*}

and, using for any $j\in\bbbn$ the expression of $\lambda_{\Lambda,\beta,j}$,
\begin{equation*}
  \mmmd_{TV}^*(\Phi_{\Lambda,\beta},\zeta_{\Lambda,\lambda})\leq2 \displaystyle\sum_{j=1}^{+\infty}\dfrac{\lambda_j^2\beta^2 Z_{\Lambda,\beta,j}^4}{1-\lambda_j\beta Z_{\Lambda,\beta,j}^2}.
\end{equation*}
Since $\lambda_j\leq1$ and $Z_{\Lambda,\beta,j}^2\leq1$, it follows that
\begin{equation*}
  \mmmd_{TV}^*(\Phi_{\Lambda,\beta},\zeta_{\Lambda,\lambda})\leq2\dfrac{\beta^2}{1-\beta}\displaystyle\sum_{j=1}^{+\infty}\lambda_jZ_{\Lambda,\beta,j}^2,
\end{equation*}

and the computation of the right hand side provides:
\begin{eqnarray*}
  2\dfrac{\beta^2}{1-\beta}\displaystyle\sum_{j=1}^{+\infty}\lambda_jZ_{\Lambda,\beta,j}^2&=&2\dfrac{\beta^2}{1-\beta}\int_{\beta^{-\frac{1}{d}}\Lambda}\displaystyle\sum_{j=1}^{+\infty}\lambda_j |h_j(x)|^2\d x \\
                                                                                          &=&2\dfrac{\beta^2}{1-\beta}\int_{\beta^{-\frac{1}{d}}\Lambda} \lambda \d x \\
                                                                                          &=&2\dfrac{\beta}{1-\beta} \lambda|\Lambda|,
\end{eqnarray*}

which concludes the proof.
\end{proof}
\begin{proof}[Proof of Theorem \ref{Supertheo6}]
  By Theorem \ref{Gibbsprop1}, the point process $\Phi$ has a Papangelou
  intensity $c$ given for any $x,x_1,\dots,x_k\in\XX$ by:
$$c(x,\{x_1,\dots,x_k\})=e^{-\theta(\Psi_1(x)+\sum_{i=1}^k\Psi_2(x,x_i))}.$$
Then,
\begin{multline*}
  |c(x,\{x_1,\dots,x_k\})-e^{-\theta\Psi_1(x)}|=e^{-\theta\Psi_1(x)}|e^{-\theta\sum_{i=1}^k\Psi_2(x,x_i)}-1|\\
  \leq e^{-\theta\Psi_1(x)}(1-e^{-\theta k\epsilon}),
\end{multline*}
and, since for any $x\geq0$, $1-e^{-x}\leq x$, one gets
$$1-\Esp[e^{-\theta|\Phi|\epsilon}]\leq\Esp[\theta|\Phi|\epsilon]=\theta\epsilon\Esp[|\Phi|].$$
Moreover, by Theorem \ref{Papangelouprop2},
\begin{equation*}
  \Esp[|\Phi|]=\int_\XX \Esp[c(x,\Phi)]\d x=\int_\XX \Esp[e^{-\theta(\Psi_1(x)+\sum_{y\in\Phi}\Psi_2(x,y))}]\d x,
\end{equation*}

and, since $\Psi_2\geq0$, it follows that
\begin{equation*}
  \Esp[|\Phi|]\leq\int_\XX e^{-\theta\Psi_1(x)} \d x=M(\XX).
\end{equation*}
As a consequence, by Theorem \ref{Steinprop},
\begin{eqnarray*}
  \mmmd_{TV}^*(\Phi,\zeta_M)&\leq&\int_\XX\Esp[|c(x,\Phi)-e^{-\theta\Psi_1(x)}|]\d x \\
                            &=&(M(\XX))^2\theta\epsilon.
\end{eqnarray*}
The proof is thus complete.
\end{proof}
\begin{proof}[Proof of Lemma \ref{lem1Supertheo4}]
  This equation is directly deduced from the definition of the Polish distance
  $\mmmd_P$.
\end{proof}
\begin{proof}[Proof of Lemma \ref{lem2Supertheo4}]
We want to find  a measurable function $c:\XX\times\N_\XX\to\bbbr_+$
  verifying, for any measurable function $u:\XX\times\N_\XX\to\bbbr_+$,
\begin{equation*}
  \Esp\Big[\displaystyle\sum_{x\in{p}\circ\varphi}u(x,({p}\circ\varphi)\setminus x)\Big]=\int_\XX \Esp[c(x,{p}\circ\varphi)u(x,{p}\circ\varphi)] p(x)\varphi(\d x).
\end{equation*}
On one hand, let us compute the left hand side:
\begin{eqnarray*}
  \Esp\Big[\displaystyle\sum_{x\in{p}\circ\varphi}u(x,({p}\circ\varphi)\setminus x)\Big]&=&\Esp\Big[\displaystyle\sum_{\eta\subset\varphi}\textbf{1}_{\{\eta={p}\circ\varphi\}}\displaystyle\sum_{x\in\eta}u(x,\eta\setminus x)\Big] \\
                                                                                        &=&\displaystyle\sum_{\eta\subset\varphi} \P(\eta={p}\circ\varphi) \displaystyle\sum_{x\in\eta}u(x,\eta\setminus x).
\end{eqnarray*}
Then, since for any $\eta\subset\varphi$,
\begin{equation*}
  \P(\eta={p}\circ\varphi)=\Big(\displaystyle\prod_{t\in\eta}p(t)\Big) \Big(\displaystyle\prod_{s\in\varphi\setminus\eta}(1-p(s))\Big),
\end{equation*}

it follows that
\begin{equation*}
  \Esp\Big[\displaystyle\sum_{x\in{p}\circ\varphi}u(x,({p}\circ\varphi)\setminus x)\Big]=\displaystyle\sum_{\eta\subset\varphi} \Big(\displaystyle\prod_{t\in\eta}p(t)\Big) \Big(\displaystyle\prod_{s\in\varphi\setminus\eta}(1-p(s))\Big) \displaystyle\sum_{x\in\eta}u(x,\eta\setminus x).
\end{equation*}
Hence,
\begin{multline*}
  \Esp\Big[\displaystyle\sum_{x\in{p}\circ\varphi}u(x,({p}\circ\varphi)\setminus x)\Big] =\\
  {=\displaystyle\sum_{x\in\varphi} \displaystyle\sum_{\substack{\eta\subset\varphi \\ x\in\eta}} \Big(\displaystyle\prod_{t\in\eta}p(t)\Big) \Big(\displaystyle\prod_{s\in\varphi\setminus\eta}(1-p(s))\Big) u(x,\eta\setminus x)} \\
  {=\displaystyle\sum_{x\in\varphi} \displaystyle\sum_{\eta\subset\varphi\setminus\{x\}} \Big(\displaystyle\prod_{t\in\eta}p(t)\Big) p(x) \Big(\displaystyle\prod_{s\in\varphi\setminus\eta}(1-p(s))\Big) \dfrac{1}{1-p(x)} u(x,\eta),}\\
\end{multline*}
and finally
\begin{multline*}
  \Esp\Big[\displaystyle\sum_{x\in{p}\circ\varphi}u(x,({p}\circ\varphi)\setminus
  x)\Big]\\
  =\displaystyle\sum_{x\in\varphi} \displaystyle\sum_{\eta\subset\varphi} \textbf{1}_{\{x\notin\eta\}} \Big(\displaystyle\prod_{t\in\eta}p(t)\Big) \Big(\displaystyle\prod_{s\in\varphi\setminus\eta}(1-p(s))\Big) \dfrac{p(x)}{1-p(x)} u(x,\eta).
\end{multline*}
On the other hand, for a given measurable function
$c:\XX\times\N_\XX\to\bbbr_+$,
\begin{multline*}
  \int_\XX \Esp[c(x,{p}\circ\varphi)u(x,{p}\circ\varphi)] p(x)\varphi(\d x)= \\
  \begin{aligned}
    &=\int_\XX \Esp\Big[\displaystyle\sum_{\eta\subset\varphi} \textbf{1}_{\{\eta={p}\circ\varphi\}} c(x,\eta)u(x,\eta)\Big] p(x)\varphi(\d x) \\
    &=\int_\XX \displaystyle\sum_{\eta\subset\varphi} \P(\eta={p}\circ\varphi) c(x,\eta)u(x,\eta) p(x)\varphi(\d x) \\
    &=\displaystyle\sum_{x\in\varphi} p(x) \displaystyle\sum_{\eta\subset\varphi} \Big(\displaystyle\prod_{t\in\eta}p(t)\Big) \Big(\displaystyle\prod_{s\in\varphi\setminus\eta}(1-p(s))\Big) c(x,\eta)u(x,\eta),\\
  \end{aligned}
\end{multline*}

where the last equality is obtained by the expression of
$\P(\eta=p\circ\varphi)$ as above. This implies that
\begin{multline*}
  \int_\XX \Esp[c(x,{p}\circ\varphi)u(x,{p}\circ\varphi)] p(x)\varphi(\d x) = \\
  = \displaystyle\sum_{x\in\varphi} \displaystyle\sum_{\eta\subset\varphi}
  \Big(\displaystyle\prod_{t\in\eta}p(t)\Big)
  \Big(\displaystyle\prod_{s\in\varphi\setminus\eta}(1-p(s))\Big) p(x)
  c(x,\eta)u(x,\eta),
\end{multline*}
and the result is got by identification.
\end{proof}
\begin{proof}[Proof of Lemma \ref{lem3Supertheo4}]
  For any $\varphi\in N_\XX$, let $\zeta_{p\varphi}$ be a Poisson point process
  with intensity measure $p\varphi$. By Theorem \ref{Steinprop},
\begin{equation*}
  \mmmd_{TV}^*(\P_{p\circ\varphi},\P_{\zeta_{p\varphi}}) \leq \int_\XX \Esp[|c(x,p\circ\varphi)-1|] p(x)\varphi(\d x),
\end{equation*}

where $c$ is a version of the Papangelou intensity of $p\circ\varphi$ with
respect to $p\varphi$. An expression of $c$ is given by Lemma
\ref{lem2Supertheo4}, and it follows that
\begin{equation*}
  \mmmd_{TV}^*(\P_{p\circ\varphi},\P_{\zeta_{p\varphi}}) \leq \int_\XX \Esp\Big[\Big|\textbf{1}_{\{x\in\varphi\setminus p\circ\varphi\}}\dfrac{1}{1-p(x)}-1\Big|\Big] p(x)\varphi(\d x).
\end{equation*}
Hence, the computation of the right hand side in the last inequality aims to
obtain that
\begin{eqnarray*}
  \mmmd_{TV}^*(\P_{p\circ\varphi},\P_{\zeta_{p\varphi}})&\leq& \int_\XX \Big(\Esp\Big[\textbf{1}_{\{x\in p\circ\varphi\}}\Big|\textbf{1}_{\{x\in\varphi\setminus p\circ\varphi\}}\dfrac{1}{1-p(x)}-1\Big|\Big]\\
                                                        &&+ \Esp\Big[\textbf{1}_{\{x\notin p\circ\varphi\}}\Big|\textbf{1}_{\{x\in\varphi\setminus p\circ\varphi\}}\dfrac{1}{1-p(x)}-1\Big|\Big]\Big) p(x)\varphi(\d x) \\
                                                        &=& \int_\XX \Big(p(x) + (1-p(x))\big(\dfrac{1}{1-p(x)}-1\big) \Big) p(x)\varphi(\d x)\\
                                                        &=& 2\displaystyle\sum_{x\in \varphi}p^2(x),
\end{eqnarray*}

and we can deduce that
\begin{equation*}
  \mmmd_{TV}^*(\P_{p\circ\Phi},\P_{\Gamma_{p\Phi}}) \leq 2\Esp\Big[\displaystyle\sum_{x\in\Phi}p^2(x)\Big].
\end{equation*}
The conclusion follows.
\end{proof}
\begin{proof}[Proof of Theorem \ref{Supertheo4}]
  By the triangle inequality,
\begin{equation*}
  \mmmd_P({\Gamma_M},{p\circ\Phi})\leq \mmmd_P({\Gamma_M},{\Gamma_{p\Phi}}) + \mmmd_P({\Gamma_{p\Phi}},{p\circ\Phi}).
\end{equation*}
One one hand, by Lemma \ref{lem1Supertheo4},
\begin{equation*}
  \mmmd_P({\Gamma_M},{\Gamma_{p\Phi}})=\overline{\mmmd}_P({M},{p\Phi}).
\end{equation*}
On the other hand, since $(f_k)_{k\in\bbbn}\subset\Lip_1(\mmmd_{TV})$,
\begin{equation*}
  \mmmd_P({\Gamma_{p\Phi}},{p\circ\Phi})\leq \mmmd_{TV}^*({\Gamma_{p\Phi}},{p\circ\Phi}),
\end{equation*}

and then, by Lemma \ref{lem3Supertheo4},
\begin{equation*}
  \mmmd_P({\Gamma_{p\Phi}},{p\circ\Phi})\leq 2\Esp\Big[\displaystyle\sum_{x\in\Phi}p^2(x)\Big],
\end{equation*}
which concludes the proof.
\end{proof}

\end{document}